\documentclass{amsart}
\usepackage{amssymb,epsfig,ifthen}
\usepackage{ifpdf}
\usepackage[section]{placeins}

\makeatletter
\AtBeginDocument{%
  \expandafter\renewcommand\expandafter\subsection\expandafter{%
    \expandafter\@fb@secFB\subsection
  }%
}
\makeatother

\usepackage{color}
\usepackage{bm}
\usepackage{bbold}
\usepackage{mathrsfs}
\usepackage{stmaryrd}

\newtheorem{theorem}{Theorem}[section]
\newtheorem{lemma}[theorem]{Lemma}
\newtheorem{proposition}[theorem]{Proposition}
\newtheorem{corollary}[theorem]{Corollary}

\theoremstyle{definition}

\theoremstyle{remark}
\newtheorem{remark}[theorem]{Remark}

\numberwithin{equation}{section}

\newcommand{\new}[1]{{\color{red}{#1}}}
\newcommand{\blue}[1]{{\color{blue}{#1}}}

\newcommand{\1}{\mathbb 1}
\newcommand{\R}{\mathbb R}
\newcommand{\N}{\mathbb N}

\newcommand{\tria}[1][]{{\mathcal T}_{#1}} 

\newcommand{\nodes}{{\mathcal V}_{\tria}}
\newcommand{\intnodes}{\nodes^\mathrm{int}}
\newcommand{\bdrnodes}{\nodes^\mathrm{bdr}}

\newcommand{\edges}[1][]{{\mathcal E}_{#1}}
\newcommand{\intedges}[1][]{{\mathcal E}^\mathrm{int}_{#1}}
\newcommand{\bdredges}[1][]{{\mathcal E}^\mathrm{bdr}_{#1}}

\newlength{\leftstackrelawd}
\newlength{\leftstackrelbwd}
\def\leftstackrel#1#2{\settowidth{\leftstackrelawd}%
{${{}^{#1}}$}\settowidth{\leftstackrelbwd}{$#2$}%
\addtolength{\leftstackrelawd}{-\leftstackrelbwd}%
\leavevmode\ifthenelse{\lengthtest{\leftstackrelawd>0pt}}%
{\kern-.5\leftstackrelawd}{}\mathrel{\mathop{#2}\limits^{#1}}}

\DeclareMathOperator{\ran}{range}
\DeclareMathOperator{\vol}{vol}
\DeclareMathOperator{\meas}{meas}
\DeclareMathOperator*{\argmin}{argmin}

\DeclareMathOperator{\supp}{supp}

\DeclareMathOperator{\Span}{span}
\DeclareMathOperator{\divv}{div}
\DeclareMathOperator{\curl}{curl}

\DeclareMathOperator{\diam}{diam}
\DeclareMathOperator{\osc}{osc}

\renewcommand{\P}{\mathcal P}
\newcommand{\RT}{\mathcal{RT}}

\title{On p-Robust Saturation for hp-AFEM}
\thanks{The first and fourth authors have been partially supported by the Italian
research grant Prin 2012 2012HBLYE4 ``Metodologie innovative nella modellistica
differenziale numerica" and by INdAM-GNCS. The second author has
been partially supported by the NSF grant DMS 1418994 and the Institut
Henri Poincar\'e (Paris).}

\date{\today}

\author[C.Canuto]{Claudio~Canuto}
\address{Dipartimento di Scienze Matematiche,
Politecnico di Torino,
Corso Duca degli Abruzzi 24,
I-10129 Torino, Italy} 
\email{claudio.canuto@polito.it}
\author[R.H.Nochetto]{Ricardo~H.~Nochetto}
\address{Department of Mathematics and Institute for Physical Science
and Technology, University of Maryland, College Park, Maryland 20742, USA}
\email{rhn@math.umd.edu}
\author[R.Stevenson]{Rob~Stevenson}%
\address{Korteweg-de Vries Institute for Mathematics,
University of Amsterdam,
P.O. Box 94248,
1090 GE Amsterdam, The Netherlands}
\email{r.p.stevenson@uva.nl}
\author[M.Verani]{Marco~Verani}
\address{MOX-Dipartimento di Matematica, Politecnico di Milano, P.zza Leonardo Da Vinci 32, I-20133 Milano, Italy}
\email{marco.verani@polimi.it}

\subjclass[2010]{%
65N30, 
65N12 
}

\keywords{$hp$-finite element method, adaptivity, convergence}

\begin{document}
\begin{abstract}
For the Poisson problem in two dimensions,
we consider the standard adaptive finite element loop {\bf solve}, {\bf estimate}, {\bf mark}, {\bf refine}, with 
{\bf estimate} being implemented using the $p$-robust equilibrated flux estimator, and, {\bf mark} being D\"{o}rfler marking.
As a refinement strategy we employ $p$-refinement. We investigate the question by which amount the local polynomial degree on any marked patch has to be increase in order to achieve a $p$-independent error reduction.
The resulting adaptive method can be turned into an instance optimal $hp$-adaptive method by the addition of a coarsening routine.
\end{abstract}

\maketitle
\section{Introduction} \label{Sintroduction}
Since the work of Babu{\v{s}}ka and co-workers in the 1980s, see e.g. \cite{77.55, 77.56}, it is known that for elliptic boundary value problems with sufficiently smooth coefficients and data, a proper combination of $h$-refinement and $p$-enrichment can yield a sequence of finite element solutions that converge exponentially fast to the solution.
Existing convergence results on $hp$-finite elements mainly concern
methods where an {\em a priori} decision about $h$- or $p$-improvement
is based on the decomposition of the solution into
smooth parts and known singular functions associated to corners or edges of the boundary.

In practice, one rather uses $hp$-{\em adaptive finite element
methods ($hp$-AFEMs)} driven by a posteriori error estimators.
Apart from identifying the elements where the
current approximation error is `relatively large',
and therefore need to be refined, 
for each of those elements $hp$-AFEMs have to decide whether
it is appropriate to perform either an $h$-refinement or a $p$-enrichment.
Ideally such a decision depends on the local smoothness of the exact
solution, which is however unknown.
Many proposals have been made to estimate the local smoothness from
information extracted from the computed sequence of finite
element solutions and the right-hand side.
Only a few of these methods have been proven to converge (e.g. \cite{69.1}), but none of them has been shown to yield exponential convergence rates.

In \cite{35.99} we followed a different approach: we extended the $hp$-AFEM with the $hp$-coarsening routine constructed by Binev in \cite{22.556}.
This routine is called after each sequence of adaptive
enhancement steps that reduces the error with a fixed, sufficiently large factor.
The application of coarsening generally makes the error larger, but it yields an $hp$-partition that is {\em instance optimal}.
That is, the best approximation error from the associated $hp$-finite element space is at most a constant factor larger than that from any 
$hp$-finite element space with a dimension that is at most a constant
factor smaller.
In particular, this means that if the solution can be approximated from $hp$-finite element spaces with an exponential rate, then the sequence
of approximations produced by our $hp$-AFEM converges exponentially to the solution (see \cite{35.99} for more details).

What remains is to specify a method that, between every two consecutive calls of coarsening, reduces the error with a sufficiently large fixed factor, and that
runs at an acceptable cost $F$, measured in terms
of floating point operations.
In doing so, we simply assume that the arising linear systems are solved exactly (whereas an inexact solve within a sufficiently small relative tolerance would be sufficient),
and moreover make the admittedly disputable assumption
that this solving is achieved at a cost that is (uniformly) proportional to the dimension of the system.
In this setting, the ideal situation would be that $F$ is proportional to
the dimension $N$ of the finite element space.
In view of the envisaged exponential decay rate
$\exp(-\eta N^\alpha)$ of the error,
a polynomial cost $F=\Lambda N^k$, with $\Lambda>0$ and $k>1$, could still be regarded as
acceptable because the error decay rate in terms of $F$
would still be exponential, namely $\exp(-\eta \Lambda^{-\alpha/k} F^{\alpha/k})$.
Note, however, the suboptimal exponent $\alpha/k < \alpha$.

In \cite{35.99}, we examined the Poisson problem in two space
dimensions
\begin{equation}\label{poisson}
-\triangle u = f \quad\text{ in } \Omega,
\qquad
u = 0 \quad \text{ on } \partial \Omega,
\end{equation}
and used as the `error reducer' a number of $h$-refinements of the elements
that were selected by bulk chasing, also known as D\"{o}rfler marking,
using the a posteriori error estimator of Melenk and Wohlmuth
\cite{MW01}. This estimator is sensitive to the polynomial degree
that generally varies over the partition.
To account for this deficiency, we showed that a number of iterations
that grows as $p^{2+\epsilon}$, here with $p$ denoting the maximal polynomial degree, 
is sufficient to achieve a fixed error reduction. Unfortunately, this
does not lead to an acceptable cost $F$. To see this, consider the 
extreme situation of having a partition $\tria$ consisting of one
single element with polynomial degree $p$, whence with finite
element space dimension $N$ proportional to $p^2$. Since we have
to perform $p^{2+\epsilon}\approx N^{1+\epsilon/2}$ steps, and each bisection step
multiplies the space dimension by $2$, we infer that the
total cost would grow as $F \approx N 2^{N^{1+\epsilon/2}}$ which is
exponential rather than polynomial.

In the present work, we consider D\"{o}rfler marking based on the
equilibrated flux estimator, which has been shown to be $p$-robust by
Braess, Pillwein and Sch\"{o}berl in \cite{33}.
We use this estimator to mark patches of elements around vertices
(stars) and to execute $p$-refinements exclusively.
To state the type of question we are after, for the ease of
presentation let us for the moment assume that the
polynomial degree is uniform over the mesh and equal to $p$.
In such a case, we denote by $u_p$ the Galerkin solution
and pose the following question:

\vskip0.2cm
\begin{minipage}{0.9\linewidth}
{\it Which increase $q = q(p)$ of polynomial degree 
is able to reduce the error by some fixed factor
$\alpha<1$ independent of $p$, namely}
\end{minipage}
\begin{equation}\label{contraction}
  \|\nabla(u-u_{p+q})\|_{L_2(\Omega)} \le \alpha
  \|\nabla(u-u_p)\|_{L_2(\Omega)}.
\end{equation}
Since the functions $\nabla(u-u_{p+q})$ and
$\nabla(u_{p+q} - u_p)$ are orthogonal in $L_2(\Omega)$ 
due to Galerkin orthogonality, \eqref{contraction} is equivalent to
the {\it saturation} property
\begin{equation}\label{saturation}
  \beta \|\nabla(u - u_p)\|_{L_2(\Omega)} \le 
  \|\nabla(u_{p+q} - u_p)\|_{L_2(\Omega)} 
\end{equation}
with a constant $0<\beta<1$ independent of $p$; consequently
$\alpha = \sqrt{1-\beta^2}$.

In this paper we develop a local version of \eqref{saturation},
written on stars in terms of negative norms of residuals, and show
that combined with D\"orfler marking of stars it yields a $p$-robust
contraction property for $hp$-AFEMs.
Our main result is a reduction of such a local saturation property to
three Poisson problems on the reference triangle
with interior or boundary sources which are polynomials of either
degree $p-1$ or $p$.
Numerical computations for these auxiliary problems suggest that
uniform saturation is achieved by increasing the polynomial degree $p$
by an additive quantity $q(p)=\lceil \lambda p\rceil$ for any fixed
constant $\lambda > 0$, whereas an increase of the form $q(p)=m$
for some constant $m \in \N$ seems insufficient.
Since multiplying $p$ by a constant factor
multiplies the dimension of the local finite element space by not
more than a constant factor, this type of $p$-enrichment meets
the desired cost condition. In fact, it leads to an ideal
cost proportional to the space dimension.

Saturation conditions such as \eqref{contraction} have been used in a
posteriori error analysis for $h$-refinement. We refer to
\cite{BW85,BEK:96}, where \eqref{contraction} is assumed to hold for
piecewise linear and quadratic finite elements, and to \cite{Noch93}
which shows that \eqref{contraction} can be removed as an assumption. The relation
between \eqref{contraction} and the relative size of data oscillation
for piecewise linear and quadratic finite elements
have been examined by D\"orfler and Nochetto \cite{DN02}. Our
abstract $hp$-AFEM of Sect. \ref{SRconvergence} bears a
resemblance to the construction in \cite{DN02}.

This work is organized as follows. In Sect.~\ref{Smodel} we describe
the model problem. In Sect.~\ref{Slocalizedresiduals} we give error
estimators in terms of
negative norms of localized residuals that are defined on stars. These error estimators are shown to be equivalent to computable estimators in Sect.~\ref{Sequivestimators}.
Under a saturation assumption on the marked stars, that mimics
\eqref{saturation} locally but is expressed in terms of negative norms
of residuals, we prove in Sect.~\ref{SRconvergence} a $p$-robust
contraction property for an abstract $hp$-AFEM.
In Sect.~\ref{Ssaturation}, these saturation assumptions on stars are reduced to
the question of saturation for three Poisson problems on a reference triangle with polynomial source terms of degree $p$.
Finally, in Sect.~\ref{Sconstants}, saturation constants are computed numerically for a range of $p$'s, and trial spaces of degree $p+q(p)$ for various choices of the function $q$.
A conclusion is drawn in Sect.~\ref{Sconclusion}.

In this work, by $C \lesssim D$ we will mean that $C$ can be bounded by a multiple of $D$, independently of parameters which C and D may depend on, where in particular we have in mind the mesh partition, and the polynomial degrees.
Obviously, $C \gtrsim D$ is defined as $D \lesssim C$, and $C\eqsim D$ as $C\lesssim D$ and $C \gtrsim D$.

For any measurable $\omega \subset \R^n$,  with $\langle\cdot,,\cdot\rangle_\omega$ and $\|\cdot\|_\omega$ we denote the $L_2(\omega)$- or $L_2(\omega)^n$-inner product and norm, respectively.
For any closed subspace of $H^1(\omega)$ on which
$\|\nabla\cdot\|_\omega$ is equivalent to $\|\cdot\|_{H^1(\omega)}$
(typically, subspaces defined by vanishing traces on non-negligible
parts of $\partial\omega$, or by vanishing mean-values), we think
always of this subspace as being equipped with
$\|\nabla\cdot\|_\omega$, and so its dual as being equipped with the
resulting dual norm.

\section{Model problem} \label{Smodel}

We consider the Poisson problem \eqref{poisson}
in a polygon $\Omega \subset \R^2$, and for a $f \in L_2(\Omega)$.
Our results can easily be generalized to other boundary conditions,
and elliptic operators $-\divv (A \nabla\cdot)$ provided $A=A^\top>0$ is piecewise
constant w.r.t. to any partition $\tria$ encountered by the $hp$-AFEM.
Generalization to three space dimensions is likely possible using an
extension, announced in \cite{309}, of the proof  in \cite{33}  of
$p$-robustness of equilibrated residual a posteriori error estimators
from two to three dimensions.

Let $U_{\tria} \subset H^1_0(\Omega)$ be a space of continuous
piecewise polynomials, of variable degree $p_{\tria}$,
w.r.t. a conforming partition $\tria$ of $\Omega$ into triangles.
We assume that $\tria$ belongs to a class of uniformly shape regular triangulations.
For $T \in \tria$, we define
\begin{equation}\label{elem-degree}
p_T=p_{\tria,T}
\end{equation}
as the smallest integer such that the restriction to $T$ of any function in $U_{\tria}$ is in $\P_{p_T}(T)$.
Note that the combination of a possible variable degree, and
$U_{\tria} \subset H^1_0(\Omega)$ generally prevents $U_{\tria}|_T$
from being the complete space $\P_{p_T}(T)$. We impose, however, that
in any case $U_{\tria} \supseteq H^1_0(\Omega) \cap \prod_{T \in \tria} \P_1(T)$,
namely that continuous piecewise affine functions are contained in $U_{\tria}$.

Let $u$ and $u_{\tria}$ denote the exact solution of our boundary value problem and its Galerkin approximation from $U_{\tria}$, respectively.

Let $\nodes$ ($\edges[\tria]$) denote the collection of vertices (edges) of $\tria$ subdivided into the interior vertices (edges) $\intnodes$ ($\intedges[\tria]$) and boundary vertices (edges) $\bdrnodes$ ($\bdredges[\tria]$).
For any $e \in \edges[\tria]$, $\vec{n}_e$ stands for a unit vector normal to $e$. The operator $\llbracket \,\rrbracket$ yields the jump, in the direction of $\vec{n}_e$, of the traces of the argument from the two triangles that share $e \in \intedges[\tria]$, and the actual trace for $e \in \bdredges[\tria]$.

\section{Localizing the residual to stars} \label{Slocalizedresiduals}
For $a \in \nodes$, let $\psi_a=\psi_{\tria,a}$ denote the continuous
piecewise linear `hat' function w.r.t. $\tria$ that has value $1$ at
$a$ and that is zero at all other vertices.
Let $\omega_a=\omega_{\tria,a}$ denote the star centered at $a$,
namely the interior of the union of
the $T \in \tria$ that share the vertex $a$; note that
$\overline{\omega}_a = \supp \psi_a$. Let
$\tria[a]=\tria[\tria,a]$ denote the triangulation $\tria$ restricted to $\omega_a$, and let $\edges[a]=\edges[\tria,a]$ ($\intedges[a]$, $\bdredges[a]$ ) denote the collection of its (interior ($a \in e$), boundary ($a \not\in e$)) edges.
We define
\begin{equation} \label{pa}
p_a=p_{a,\tria}:=\max_{T \in \omega_a} p_T, \qquad \vec{p}_a=\vec{p}_{a,\tria}=(p_T)_{T \in \omega_a}.
\end{equation}

For $a \in \nodes$, we define
$$
H_\ast^1(\omega_a):=\left\{
\begin{array}{ll}
\{v \in H^1(\omega_a)\colon \langle v,\1\rangle_{\omega_a}=0\} & a \in \intnodes,\\
\{v \in H^1(\omega_a)\colon v=0 \text{ on } \partial\omega_a \cap \partial\Omega\} & a \in \bdrnodes,
\end{array}
\right.
$$
equipped with $\|\nabla \cdot\|_{\omega_a}$.

For $v \in H^1(\Omega)$, we define the global and localized residual functionals by
\begin{align*}
r_{\tria}(v)&:=\langle f,v\rangle_{\Omega}-\langle \nabla u_{\tria},\nabla v \rangle_{\Omega},\\
r_a(v) =r_{\tria,a}(v)&:=r_{\tria}(\psi_a v)
= \sum_{T \in \tria[a]} \int_T v \psi_a (f+\triangle u_{\tria})+\sum_{e\in \intedges[a]} \int_e v \psi_a \llbracket \nabla u_{\tria}\cdot\vec{n}_e \rrbracket.
\end{align*}

\begin{proposition}[reliability and efficiency] \label{prop1}
There exists $C_1>0$ such that
$$
\|\nabla(u-u_{\tria})\|^2_{\Omega} \leq 3 \sum_{a \in \nodes} \|r_a\|^2_{H_\ast^1(\omega_a)'},\quad \|r_a\|_{H_\ast^1(\omega_a)'} \leq C_1 \|\nabla(u-u_{\tria})\|_{\omega_a} \quad(a \in \nodes).
$$
\end{proposition}

\begin{proof}
Thanks to $\psi_a \in U_{\tria}$, and $u_{\tria}$ being the Galerkin solution from $U_{\tria}$ (cf. Remark~\ref{rem2}), we have
$$
r_a(\1)=0 \quad (a \in \intnodes).
$$
Since furthermore $\sum_{a \in \nodes} \psi_a=\1$, 
for $v \in H^1_0(\Omega)$ we have
\begin{align*}
r_{\tria}(v)=\sum_{a \in \nodes} r_a(v)&=\sum_{a \in \bdrnodes} r_a(v)+\sum_{a \in \intnodes} r_a\Big(v-\frac{\int_{\omega_a} v}{\vol{\omega_a}}\1\Big)
\\
& \leq \sum_{a \in \nodes} \|r_a\|_{H^1_\ast(\omega_a)'} \|\nabla v\|_{\omega_a} \\
& \leq \sqrt{3}  \|\nabla v\|_{\Omega} \sqrt{\sum_{a\in \nodes} \|r_a\|^2_{H^1_*(\omega_a)'}},
\end{align*}
where we used that each triangle in $\tria$ is contained in at most three patches $\omega_a$.
From $\|\nabla(u-u_{\tria})\|_{\Omega} = \sup_{0 \neq v \in H^1_0(\Omega)} r_{\tria}(v)/\|\nabla v\|_{\Omega}$, we arrive at the first result.

Conversely, for $v \in H^1_*(\omega_a)$ we have that $r_a(v) \leq \|\nabla (u-u_{\tria})\|_{\omega_a}\|\nabla(\psi_a v)\|_{\omega_a}$.
By either applying the Poincar\'{e} inequality for $a \in \intnodes$ or the Friedrichs inequality for $a \in \bdrnodes$, we have 
$\|\nabla(\psi_a v)\|_{\omega_a} \lesssim \|\nabla  v\|_{\omega_a}$ (\cite[p.8]{70.8}), which yields the second result.
\end{proof}

Exactly the same proof of Proposition \ref{prop1} shows the following result:
\begin{proposition}[discrete reliability and efficiency] \label{prop3}
Let $\bar{U} \supset U_{\tria}$ be a closed subspace of
$H^1_0(\Omega)$ with Galerkin solution w.r.t. $\bar{U}$ denoted by
$\bar{u}$. We then have for all $a \in \nodes$
$$
\|\nabla(\bar{u}-u_{\tria})\|^2_{\Omega} \leq 3 \sum_{a \in \nodes} \|r_a\|^2_{(H_\ast^1(\omega_a)\cap \bar{U})'},\quad \|r_a\|_{(H_\ast^1(\omega_a)\cap \bar{U})'} \leq C_1 \|\nabla(\bar{u}-u_{\tria})\|_{\omega_a}.
$$
\end{proposition}
\noindent For completeness, here and on other places, with
$H_\ast^1(\omega_a) \cap \bar{U}$, we mean the space of functions in
$H_\ast^1(\omega_a)$ that are restriction of functions from $\bar{U}$.

Let $Q_{p,T}$ denote the $L_2(T)$-orthogonal projection onto
$\P_p(T)$, and define
\begin{equation}\label{Qtriaa}
(Q_{\tria[a]} w)|_T :=
\left\{\begin{array}{cl@{}} Q_{p_T-1,T} w & p_T \geq 2, \\ \frac{\int_T \psi_a w}{\int_T \psi_a} & p_T=1,\end{array} \right.
\end{equation}
One easily infers that $w \mapsto \frac{\int_T \psi_a w}{\int_T \psi_a}$
is an $L_2(T)$-bounded projector, with norm equal to $\sqrt{\frac{3}{2}}$.
Let $\breve{r}_a$ be the residual computed on discrete data
$Q_{\tria[a]} f$ rather than $f$:
\begin{equation} \label{2}
 \breve{r}_a(v): = \sum_{T \in \tria[a]} \int_T v \psi_a ((Q_{\tria[a]} f)|_T+\triangle u_{\tria})+\sum_{e\in \intedges[a]} \int_e v \psi_a \llbracket \nabla u_{\tria}\cdot\vec{n}_e \rrbracket.
 \end{equation}
Using the definition of $Q_{\tria[a]} f$ and that $\psi_a|_T \in \P_1(T)$, we then have
 $$
  \breve{r}_a(\1)- r_a(\1)=\sum_{T \in \tria[a]} \int_T \psi_a ((Q_{\tria[a]} f)|_T-f)=0 \quad (a \in \nodes).
 $$

 The two next results show that the norms of $r_a$ and $\breve{r}_a$
 are equal modulo local {\em data oscillation}.
 
 \begin{proposition}[discrepancy between $r_a$ and $\breve{r}_a$] \label{prop2} It holds that
 $$\|r_a- \breve{r}_a\|_{H^1_*(\omega_a)'} \lesssim \sum_{T \in \tria[a]} \diam(T) \inf_{f_{p_T-1} \in \P_{p_T-1}(T)} \|f-f_{p_T-1}\|_T,$$
 so that in particular $ \breve{r}_a=r_a$ when $f \in \prod_{T \in \tria} \P_{p_T-1}(T)$.
 \end{proposition}
 
 \begin{proof}
 By applying the Poincar\'{e} inequality for $a \in \intnodes$ or the Friedrichs inequality for $a \in \bdrnodes$, we infer that
 \begin{align*}
\|r_a- \breve{r}_a\|_{H^1_*(\omega_a)'}
&=
\sup_{0 \neq v \in H^1_*(\omega_a)} \frac{\sum_{T \in \tria[a]} \int_T v \psi_a (f-(Q_{\tria[a]} f)|_T)}{\|\nabla v\|_{\omega_a}}\\
& \lesssim \diam(\omega_a)\sqrt{ \sum_{T \in \tria[a]} \|f-(Q_{\tria[a]} f)|_T\|_T^2}\\
&\lesssim \sum_{T \in \tria[a]} \diam(T) \inf_{f_{p_T-1} \in \P_{p_T-1}(T)} \|f-f_{p_T-1}\|_T. \qedhere
\end{align*}
 \end{proof}
 
Straightforward applications of Cauchy-Schwarz inequalities give the
following bound in terms of global data oscillation, which is
defined as follows:
\[
\osc(f,\tria) = \sqrt{\sum_{T \in \tria}\diam(T)^2 \inf_{f_{p_T-1} \in
  \P_{p_T-1}(T)} \|f-f_{p_T-1}\|_T^2 }.
\]
 
 \begin{corollary}[global discrepancy between residuals] \label{corol1}There exists a constant $C_2>0$ such that
 \begin{align*}
  \left| \sqrt{\sum_{a \in \nodes}
    \|\breve{r}_a\|_{H^1_*(\omega_a)'}^2} - \sqrt{\sum_{a \in \nodes}
    \|r_a\|_{H^1_*(\omega_a)'}^2} \right| \leq  C_2 \osc(f,\tria).
 \end{align*}
 \end{corollary}
 
 \begin{remark}[weaker dual norms] \label{rem1} Obviously the result of Proposition~\ref{prop2}, and thus that of Corollary~\ref{corol1}, is also valid when the dual norms are taken w.r.t. closed subspaces of the spaces $H^1_*(\omega_a)$, as those employed in Proposition~\ref{prop3}.
\end{remark}

\begin{remark}[minimal requirement on $u_{\tria}$] \label{rem2}
The results that were obtained in this section, and so those that are based on them as the forthcoming Proposition~\ref{Rconv}, are actually valid for {\em any} $u_{\tria} \in U_{\tria}$
that satisfies
\begin{equation} \label{linears}
\int_\Omega \nabla u_{\tria} \cdot \nabla v \,dx=\int_\Omega f v\,dx \quad (v \in H^1_0(\Omega) \cap \prod_{T \in \tria} \P_1(T)),
\end{equation}
being the property responsible for $r_a(\1)=0$ ($a \in \intnodes$).
So $u_{\tria}$ does not have to be the Galerkin solution from $U_{\tria}$.
\end{remark}

\section{$p$-robust convergence of $hp$-AFEM} \label{SRconvergence}

Let $\theta \in (0,1]$, $\sigma \in (0,1]$, $\lambda \in (0,\frac{\sigma \theta}{C_2})$
be constants.
We consider an abstract $hp$-AFEM which comprises the following
three steps between consecutive Galerkin solves:
\begin{enumerate} \renewcommand\theenumi{\roman{enumi}}
\item \label{y1}
{\it Small data oscillation}:
let global data oscillation be sufficiently small relative to the global estimator
\[
\osc(f,\tria) \leq \lambda \sqrt{\sum_{a \in \nodes}
  \|\breve{r}_a\|_{H^1_*(\omega_a)'}^2};
\]
\item \label{y2}
{\it D\"orfler marking}:
let the marked set ${\mathcal M}\subset \nodes$ satisfy
\[
\sum_{a \in {\mathcal M}} \|\breve{r}_a\|^2_{H^1_*(\omega_a)'} \geq
\theta^2 \sum_{a \in \nodes} \|\breve{r}_a\|^2_{H^1_*(\omega_a)'};
\]
\item \label{y3}
{\it Local saturation property}: let
$\bar{U} \supset U_{\tria}$ be a closed subspace of
$H^1_0(\Omega)$ that saturates the dual norm $\|\breve{r}_a\|_{H^1_*(\omega_a)'}$
for each marked star $\omega_a$
\[
\|\breve{r}_a\|_{(H^1_*(\omega_a)\cap
  \bar{U})'} \geq \sigma \|\breve{r}_a\|_{H^1_*(\omega_a)'} \quad (a \in {\mathcal M})
\]
\end{enumerate}
Condition \eqref{y3} means that enlarging the discrete space suitably
ensures local saturation on the marked stars.
This abstract $hp$-AFEM is driven by the a posteriori error indicators $\|\breve{r}_a\|_{H^1_*(\omega_a)'}$, that, however, are not computable.
In the next section, we will recalled that these indicators are {\em
uniformly} equivalent to computable quantities, which can then be
used instead.
With some obvious modifications of the constants in the error reduction
factor, the following result remains valid.

\begin{proposition}[contraction of $hp$-AFEM] \label{Rconv}
Let conditions \eqref{y1}-\eqref{y3} be enforced by the
$hp$-AFEM, and let $\bar{u} \in \bar{U}$ denote the Galerkin
  solution. Then it holds that
$$
\|\nabla(u-\bar{u})\|_\Omega \leq \sqrt{1-\frac{(\sigma \theta-C_2 \lambda)^2}{(3 C_1(1+C_2 \lambda))^2}} \,\,\|\nabla(u-u_{\tria})\|_\Omega.
$$
\end{proposition}

\begin{proof} We observe that the following chain of inequalities is valid
\begin{align*}
\sqrt{3} C_1 \|\nabla(\bar{u}-u_{\tria})\|_\Omega \quad&\leftstackrel{\text{Prop.~\ref{prop3}}}{\geq} \sqrt{\sum_{a \in \nodes} \|r_a\|_{(H^1_*(\omega_a)\cap \bar{U})'}^2} \\
&\leftstackrel{\text{Corol.~\ref{corol1}, Rem.~\ref{rem1}}}{\geq} \sqrt{\sum_{a \in \nodes} \|\breve{r}_a\|_{(H^1_*(\omega_a)\cap \bar{U})'}^2} -C_2\osc(f,\tria)\\
&\leftstackrel{\text{\eqref{y1}}}{\geq}  \sqrt{\sum_{a \in {\mathcal M}} \|\breve{r}_a\|_{(H^1_*(\omega_a)\cap \bar{U})'}^2} -C_2 \lambda \sqrt{\sum_{a \in \nodes} \|\breve{r}_a\|_{H^1_*(\omega_a)'}^2}\\
&\leftstackrel{\eqref{y3}, \eqref{y2}}{\geq}  (\sigma \theta-C_2\lambda) \sqrt{\sum_{a \in \nodes} \|\breve{r}_a\|_{H^1_*(\omega_a)'}^2}\\
&\leftstackrel{\text{Corol.~\ref{corol1}, \eqref{y1}}}{\geq} \frac{\sigma \theta-C_2\lambda}{1+C_2\lambda} \sqrt{\sum_{a \in \nodes} \|r_a\|_{H^1_*(\omega_a)'}^2}\\
&\leftstackrel{\text{Prop.~\ref{prop1}}}{\geq} \frac{\sigma \theta-C_2\lambda}{\sqrt{3}(1+C_2\lambda)} \|\nabla(u-u_{\tria})\|_\Omega.
\end{align*}
Exploiting Galerkin orthogonality
$\nabla (u-\bar{u}) \perp \nabla(\bar{u}-u_{\tria})$ finishes the proof.
\end{proof}

Proposition \ref{Rconv} is reminiscent of \cite[Theorem 1.1]{DN02}
for piecewise linear and quadratic finite elements, except that 
\eqref{y1} was expressed in terms of the error; an expression similar
to \eqref{y1} is discussed in \cite[Remark 3.4]{DN02}.
Condition \eqref{y3} was derived in \cite{DN02} upon explicitly
computing a sharp relation of jump residuals against linear and
quadratic bubbles that allows for elimination of jumps in favor of
interior residuals. A similar calculation seems intractable for general
polynomial degree.
In order to enforce \eqref{y3} for any polynomial degree, we will
seek later, in Sections \ref{Ssaturation} and \ref{Sconstants},
a function $q:\N \rightarrow \N$ such that for some constant $\sigma \in (0,1]$, 
\begin{equation} \label{saturation_task}
\|\breve{r}_a\|_{(H^1_*(\omega_a)\cap \prod_{T \in \tria[a]} \P_{p_a+q(p_a)}(T))'} \geq \sigma \|\breve{r}_a\|_{H^1_*(\omega_a)'} \quad(a \in \nodes);
\end{equation}
this is a local version of \eqref{saturation} written in terms of
residuals and thereby more practical.
Upon selecting $H_0^1(\Omega) \supset \bar{U} \supset U_{\tria}$  such that 
\begin{equation}\label{def-Ubar}
H^1_*(\omega_a)\cap \prod_{T \in \tria[a]} \P_{p_a+q(p_a)}(T) \subset H^1_*(\omega_a) \cap \bar{U}\quad (a \in {\mathcal M}),
\end{equation}
we then infer that the saturation property \eqref{y3} is valid.

\begin{remark}[role of oscillation]
As discussed in Sect.~\ref{Sintroduction}, in the setting of the $hp$-AFEM algorithm from \cite{35.99}, we need the result of Proposition~\ref{Rconv} only for the case $\osc(f,\tria)=0$.
Indeed, there the actual right-hand side has already been replaced by a piecewise polynomial approximation before moving to the error reduction step.

Since the term $\osc(f,\tria)$ is generically of higher order than $\|\nabla
(u-u_{\tria})\|_\Omega$, usually \eqref{y1} is satisfied ``automatically''
also inside other $hp$-AFEM algorithms.
In the unlikely event that initially this does not hold, it can be
enforced by global, or appropriate local $p$-enrichment that drive
$\osc(f,\tria)$ to zero even though the
right-hand side of \eqref{y1} changes with $(p_T)_{T \in \tria}$.
To see this we stress that {\em without} computing new Galerkin solutions w.r.t. to
the enlarged trial spaces (which is allowed by Remark~\ref{rem2}),
Corollary~\ref{corol1} shows that
\[
\sqrt{\sum_{a \in \nodes} \|\breve{r}_a\|_{H^1_*(\omega_a)'}^2}
\rightarrow \sqrt{\sum_{a \in \nodes} \|r_a\|_{H^1_*(\omega_a)'}^2}
\quad\text{ as } \osc(f,\tria) \downarrow 0.
\]
Since Proposition~\ref{prop1} implies that
the latter expression is equivalent to $\|\nabla(u-u_{\tria})\|$,
which is thus not affected by these additional $p$-enrichments,
we infer that \eqref{y1} is satisfied when $\osc(f,\tria)$ has been
made sufficiently small.
\end{remark}

\begin{remark} [optimality]
For $h$-AFEM, i.e. fixed polynomial degree on all triangles, it is known that R-linear convergence 
already `nearly' implies a best possible convergence rate allowed by this degree and the solution, i.e. `optimality'.
Indeed, what is furthermore needed is that the error estimator is `efficient' and `discretely reliable',
and that the cardinality of any partition created by the AFEM can be bounded on a constant multiple of the total number of marked cells
starting from the initial partition. 

In $hp$-AFEM, the question how to ensure optimal rates is much more difficult. 
At a first glance, it requires a basically optimal choice of either
$h$-refinement or $p$-enrichment in every marked cell,
which seems about impossible to realize.
In \cite{35.99}, we therefore returned to the idea, introduced in \cite{21} in the $h$-AFEM setting, of correcting possibly non-optimal earlier decisions by means of coarsening.
We showed how any R-linearly convergent $hp$-AFEM can be turned into an optimally converging $hp$-AFEM by the addition of an $hp$-coarsening routine that was developed in \cite{22.556}.
This routine, that is called after every fixed number of the R-linearly converging AFEM, replaces the current AFEM solution by 
 a quasi-optimal $hp$-approximation within a suitable tolerance. 
\end{remark}

\begin{remark}[computational cost]
In \cite{35.99}, we examined the Poisson problem in two space
dimensions and derived an error reduction property upon combining
D\"{o}rfler marking with an $h$-refinement of the marked elements
driven by the a posteriori error estimators of \cite{MW01}.
We showed that allowing a number of iterations that grow polynomially
faster than quadratically with the maximal polynomial degree
is sufficient for an error reduction with a fixed factor.
However, as already discussed in Sect. \ref{Sintroduction}, this yields
a computational cost that might increase exponentially
with the polynomial degree and is thus unacceptable in practice.

In this paper, we resort to $p$-enrichment instead and investigate the 
question \eqref{saturation}, or equivalently the amount $q(p)$ by which
the local polynomial degree $p$ must be increased for {\em one} single
iteration of D\"{o}rfler marking together with $p$-enrichment of the marked patches 
to yield an error reduction by a fixed factor.
This key question is discussed in the next three sections.
\end{remark}

\section{Equivalent computable a posteriori error indicators}  \label{Sequivestimators}
In this section we recall that the dual norm of the local residuals $\breve{r}$ are equivalent to computable quantities.
These quantities can be used to drive the $hp$-AFEM.

It is well-known that a vector field  $\vec{\tau} \in \RT_p(T)$ in the
Raviart-Thomas space $\RT_p(T)$ of order $p$ over an element $T$
is uniquely determined by the conditions $\divv\vec{\tau}=\phi_T$, and $\vec{\tau}|_e\cdot \vec{n}_{T}= \phi_e$ ($e \in \edges \cap \partial T$), when 
the $\phi_T, \phi_e$ are polynomials of degree $p$ that satisfy $\int_T \phi_T=\sum_{e \in \edges \cap \partial T} \int_e \phi_e$. 
Noting that for 
\begin{equation} \label{broken}
 \vec{\sigma}_a \in \RT_{\vec{p}_a,0}^{-1}(\tria[a]):= \Big\{\vec{\sigma} \in \prod_{T \in \tria[a]} \RT_{p_T}(T)\colon \vec{\sigma}\cdot \vec{n}_e=0 \,\,(e \in \bdredges[a])\Big\}
 \end{equation}
 and $v \in H^1(\omega_a)$ one has
 $$
 - \langle \vec{\sigma}_a,\nabla v\rangle_{\omega_a} = \sum_{T \in \tria[a]} \int_T v \divv \vec{\sigma}_a+\sum_{e \in \intedges[a] }\int_{e} v \llbracket \vec{\sigma}_a \cdot \vec{n}_e\rrbracket,
 $$
in view of \eqref{2} and $\breve{r}_a(\1)=0$ ($a \in \intnodes$) one
infers that there exist (multiple) $\vec{\sigma}_a \in \RT_{\vec{p}_a,0}^{-1}(\tria[a])$ with 
\begin{equation} \label{lift}
- \langle \vec{\sigma}_a,\nabla v\rangle_{\omega_a} =\breve{r}_a(v) \quad(v \in H_*^1(\omega_a)).
\end{equation}
In the literature, such a field $\vec{\sigma}_a$ is called {\em in equilibration} with $\breve{r}_a$.
 Obviously, for {\em any} of such $\vec{\sigma}_a$ it holds that 
 \begin{equation} \label{reliable}
 \|\breve{r}_a\|_{H^1_*(\omega_a)'} \leq \|\vec{\sigma}_a\|_{\omega_a}.
 \end{equation}
 
 The next, celebrated result shows that, up to a multiplicative constant,
 the reversed inequality is true for a suitable $\vec{\sigma}_a$
 \cite[Thm.~7]{33}.
 
 \begin{theorem}[equivalent estimator]
Let $\RT^{-1}_{p_a,0}(\tria[a])$ denote the non-conforming
Raviart-Thomas space over the star $\omega_a$ of degree $p_a$
 $$
\RT^{-1}_{p_a,0}(\tria[a]):= \Big\{\vec{\sigma} \in \prod_{T \in
  \tria[a]} \RT_{p_a,0}(T)\colon \vec{\sigma}\cdot \vec{n}_e=0 \,\,(e
\in \bdredges[a])\Big\} .
 $$
Then there holds
 $$
 \argmin_{\{\vec{\sigma}_a \in \RT_{p_a,0}^{-1}(\tria[a])\colon  \vec{\sigma}_a \text{ solves } \eqref{lift}\}} \|\vec{\sigma}_a\|_{\omega_a} \eqsim \|\breve{r}_a\|_{H^1_*(\omega_a)'}.
 $$
 \end{theorem}
 
\noindent Indeed, in \cite{33} a $\vec{\sigma}_a \in \RT_{p_a,0}^{-1}(\tria[a])$ has been constructed that satisfies \eqref{lift} with $\|\vec{\sigma}_a\|_{\omega_a} \lesssim \|\breve{r}_a\|_{H^1_*(\omega_a)'}$. Together with \eqref{reliable} this proves the theorem. 
Note that, according to \eqref{pa}, we have
\[
\RT_{\vec{p}_a,0}^{-1}(\tria[a]) \subset \RT^{-1}_{p_a,0}(\tria[a]).
\]

 An efficient computation of the $\vec{\sigma}_a \in \RT_{p_a,0}^{-1}(\tria[a])$ that solves the minimization problem was proposed in \cite{70.8}:
  Using that $u_{\tria} \in U_{\tria}$, integration-by-parts and the chain rule show that
  \begin{align*}
    \breve{r}_a(v)&=\langle \psi_a Q_{\tria[a]} f,v\rangle_{\omega_a}
    - \langle \nabla u_{\tria},\nabla (\psi_a v)\rangle_{\omega_a} \\
 &= \langle \psi_a Q_{\tria[a]} f -\nabla \psi_a \cdot \nabla u_{\tria},v\rangle_{\omega_a}-\langle \psi_a \nabla u_{\tria},\nabla v\rangle_{\omega_a}.
 \end{align*}
Noting that $\psi_a \nabla u_{\tria}\in {\RT}_{p_a,0}^{-1}(\tria[a])$, and
introducing $\vec{\zeta}_a:=\vec{\sigma}_a-\psi_a \nabla u_{\tria}$, we conclude that $\vec{\sigma}_a \in {\RT}_{p_a,0}^{-1}(\tria[a])$ solves  \eqref{lift} if and only if 
$\vec{\zeta}_a \in {\RT}_{p_a,0}^{-1}(\tria[a])$ solves
 $$
 -\langle \vec{\zeta}_a,\nabla v\rangle_{\omega_a}=\langle \psi_a Q_{\tria[a]} f -\nabla \psi_a \cdot \nabla u_{\tria},v\rangle_{\omega_a}\quad(v \in H_*^1(\omega_a)).
 $$
 Since in particular  $\vec{\zeta}_a\cdot \vec{n}_e=0$ ($e \in \bdredges[a]$), and $\psi_a Q_{\tria[a]} f -\nabla \psi_a \cdot \nabla u_{\tria} \in L_2(\omega_a)$, the latter problem is equivalent to $\divv \vec{\zeta}_a=\psi_a Q_{\tria[a]} f -\nabla \psi_a \cdot \nabla u_{\tria}$, which implies that each solution satisfies
 $$
 \vec{\zeta}_a \in  \RT_{p_a,0}(\tria[a]):= H(\divv;\omega_a) \cap \RT_{p_a,0}^{-1}(\tria[a]).
 $$
 
 The problem of finding $\vec{\zeta}_a \in  \RT_{p_a,0}(\tria[a])$
 with $\divv \vec{\zeta}_a = \psi_a Q_{\tria[a]} f-\nabla \psi_a \cdot
 \nabla u_{\tria}$ and {\em minimal} $\|\vec{\zeta}_a+\psi_a \nabla
 u_{\tria}\|_{\omega_a}$ (i.e., minimal
 $\|\vec{\sigma}_a\|_{\omega_a}$)
 reduces to the following saddle point problem: find
 the pair $\vec{\zeta}_a \in  \RT_{p_a,0}(\tria[a])$, and
 $$
 r_a \in {\mathcal Q}_{p_a}(\tria[a]):=\left\{\begin{array}{cc} 
 \{q \in \prod_{T \in \tria[a]}\P_{p_a}(T) \colon \langle q,\1\rangle_{\omega_a}=0\} & a \in \intnodes, \\
\prod_{T \in \tria[a]}\P_{p_a}(T) & a \in \bdrnodes,
\end{array}
\right.
$$
such that
\begin{equation} \label{mixed}
\left\{
\begin{aligned}
\langle \vec{\zeta}_a,\vec{\tau}_a\rangle_{\omega_a}+\langle \divv \vec{\tau}_a,r_a\rangle_{\omega_a}&= -\langle \psi_a \nabla u_{\tria},\vec{\tau}_a\rangle_{\omega_a} & \hspace*{-1em} (\vec{\tau}_a \in \RT_{p_a,0}(\tria[a])),\\
\langle \divv \vec{\zeta}_a,q_a\rangle_{\omega_a}&=\langle \psi_a Q_{\tria[a]} f-\nabla \psi_a \cdot \nabla u_{\tria},q_a\rangle_{\omega_a} & \hspace*{-1em}(q_a\in {\mathcal Q}_{p_a}(\tria[a])).
\end{aligned}
\right.
\end{equation}

\begin{remark}[avoiding $Q_{\tria[a]}$]
The computation of the projection involving $Q_{\tria[a]}$ can be avoided by a slightly different definition
of $\breve{r}_a$ in \eqref{2}: If, for each $T \in \tria[a]$, we replace the term $\psi_a (Q_{\tria[a]} f)|_T$ by $Q_{p_a,T}(\psi_a  f)$, then all statements obtained so far remain valid, but in \eqref{mixed} the term $\psi_a Q_{\tria[a]} f$ would read as $\hat{Q}_{\tria[a]} (\psi_a f)$, with 
$(\hat{Q}_{\tria[a]} w)|_T:=Q_{p_a,T} w$; recall that
$Q_{p_a,T}$ is defined in \eqref{Qtriaa}.
Since $\ran (I-\hat{Q}_{\tria[a]}) \perp {\mathcal Q}_{\P_a}(\tria[a])$, we infer that in 
that case  $\hat{Q}_{\tria[a]} \psi_a f$ could simply be
replaced by $\psi_a f$.

The reason why we have nevertheless chosen our current
definition \eqref{2}  is that
it yields an $\breve{r}_a$ of the form 
\begin{equation} \label{residual_form}
 \breve{r}_a(v): = \sum_{T \in \tria[a]} \int_T v \psi_a 
 \phi_T+\sum_{e\in \intedges[a]} \int_e v \phi_e \quad (v \in H^1(\omega_a)),
\end{equation}
for some $\phi_T \in P_{p_T-1}(T)$ and $\phi_e \in P_{p_a}(e)$, and thus with 
$\psi_a \phi_T$ being a polynomial of degree $p_T$ that vanishes at $\partial \omega_a$. With the alternative definition, that allows for a slightly simpler solution of the mixed system \eqref{mixed}, the form of $\breve{r}_a(v)$ would be similar, except that $\psi_a \phi_T$ would read as
a polynomial of degree $p_a$, without boundary conditions.
Potentially, the absence of these boundary conditions makes our task of ensuring the saturation property \eqref{y3} more difficult.
(Actually, in \eqref{residual_form} also $\phi_e$ could be read as a product of a polynomial of degree $p_a-1$ and $\psi_a$, i.e., as 
a polynomial of degree $p_a$ that vanishes at $\partial \omega_a$, but we will not be able to benefit from this extra property.)
\end{remark}

\begin{remark}[alternative mixed FEMs]
Instead of using Raviart-Thomas elements, computable indicators can equally well be defined in terms of Brezzi-Douglas-Marini
or Brezzi-Douglas-Duran-Fortin mixed finite elements.
 \end{remark}

\section{Reducing the saturation problem from a star to a triangle} \label{Ssaturation}
We recall the task \eqref{saturation_task} of finding a function
$q:\N \rightarrow \N$ such that for some constant $\sigma \in (0,1]$, 
$$
\|\breve{r}_a\|_{(H^1_*(\omega_a)\cap \prod_{T \in \tria[a]} \P_{p_a+q(p_a)}(T))'} \geq \sigma \|\breve{r}_a\|_{H^1_*(\omega_a)'} \quad(a \in \nodes).
$$
In this section, we reduce this task on patches to similar tasks on a single `reference'  triangle
$\check{T}$, with edges $\check{e}_1$, $\check{e}_2$, and $\check{e}_3$.
We make use of the following two lemmas.

\begin{lemma}[$p$-robust polynomial inverse of the divergence] \label{internal} For $\phi_{\check{T}} \in \P_p(\check{T})$ there exists a $\vec{\sigma}_{\check{T}} \in \RT_p(\check{T})$ with
$$
\divv \vec{\sigma}_{\check{T}}= \phi_{\check{T}} \quad\text{and}\quad \|\vec{\sigma}_{\check{T}}\|_{\check{T}} \lesssim \|\phi_{\check{T}}\|_{H^1_0(\check{T})'}.
$$
\end{lemma}
\noindent
This lemma, formulated as a conjecture in \cite{33}, was later
proved by Costabel and McIntosh \cite{45.496};
see also the `note added to proof' following \cite[Conjecture 6]{33}.
The following lemma was shown by Demkowicz, Gopalakrishnan, and
Sch{\"o}berl \cite[Thm. 7.1]{64.145}.

\begin{lemma}[$p$-robust Raviart-Thomas extension] \label{boundary}
  Given $\phi \in L_2(\partial
  \check{T})$ such that $\phi|_{\check{e}_i} \in \P_p(\check{e}_i)$
  and $\int_{\partial \check{T}} \phi =0$, there exists a
  $\vec{\sigma}_{\check{T}} \in \RT_p(\check{T})$ with
  $\vec{\sigma}_{\check{T}} \cdot \vec{n}_{\check{T}} = \phi$, $\divv
  \vec{\sigma}_{\check{T}}=0$, and
  \[
  \|\vec{\sigma}_{\check{T}}\|_{\check{T}} \lesssim
  \inf_{\{\vec{\tau}_{\check{T}} \in H(\divv;\check{T})\colon \divv
    \vec{\tau}_{\check{T}}=0,\,\vec{\tau}_{\check{T}} \cdot
    \vec{n}_{\check{T}} = \phi\}}
  \|\vec{\tau}_{\check{T}}\|_{\check{T}}.
  \]
\end{lemma}

The announced reduction of the saturation problem is given by the following theorem.
For a Lipschitz domain $\Omega \subset \R^n$, and a $\Gamma \subset
\partial\Omega$ with $\meas(\Gamma)>0$, here and in the following we
use the notation $H^1_{0,\Gamma}(\Omega)$ to denote the closure in
$H^1(\Omega)$ of the space of smooth functions on $\overline{\Omega}$
that vanish at $\Gamma$. We now prove \eqref{saturation_task}.

\begin{theorem}[reduction of $p$-robust saturation property] \label{thm2}
Let us introduce the following three constants on the reference
  triangle $\check{T}$:
\begin{align}
 C^{(1)}_{p,q}&:=\sup_{0 \neq \phi \in \P_{p-1}(\check{T})} \frac{\|\psi \phi\|_{H_{0,\check{e}_1 \cup \check{e}_2}^1(\check{T})'}}{\|\psi \phi\|_{(H_{0,\check{e}_1 \cup \check{e}_2}^1(\check{T})\cap \P_{p+q})'} },
\end{align}
with $\psi \in \P_1(\check{T})$ defined by $\psi(\check{e}_1 \cap \check{e}_2)=1$, $\psi(\check{e}_3)=0$;
\begin{align}
C^{(2)}_{p,q}&:=\sup_{0 \neq \phi \in \P_p(\check{e}_1)} \frac{\|v \mapsto \int_{\check{e}_1} \phi v\|_{H^1_{0,\check{e}_2}(\check{T})'}}{\|v \mapsto \int_{\check{e}_1} \phi v\|_{(H^1_{0,\check{e}_2}(\check{T}) \cap \P_{p+q})'}};\\
C^{(3)}_{p,q}&:=\sup_{\{0 \neq \phi \in \prod_{i=1}^3 \P_p(\check{e}_i)\colon \int_{\partial \check{T}}\phi =0\}} \frac{\|v \mapsto \int_{\partial \check{T}} \phi v\|_{H_*^1(\check{T})'}}{\|v \mapsto \int_{\partial \check{T}} \phi v\|_{(H_*^1(\check{T}) \cap \P_{p+q})'}}.
\end{align}
If for some function $q\colon\N \rightarrow \N$ the quantity
\[
\check{C}:=\displaystyle{\sup_{p \in \N}}
  \max\Big(C^{(1)}_{p,q(p)},C^{(2)}_{p,q(p)},C^{(3)}_{p,q(p)}\Big)
\]
is finite, then there exists a constant $\sigma$ depending on
$\check{C}$ such that
\begin{equation}\label{red-sat}
 \sigma \|\breve{r}_a\|_{H_\ast^1(\omega_a)'} \le \|\breve{r}_a\|_{(H_\ast^1(\omega_a)\cap\prod_{T \in \tria[a]} \P_{p_a+q(p_a)}(T))'} \quad(a \in \nodes).
 \end{equation}
\end{theorem}

\begin{proof} This proof consists of parts (A) and (B) below.
It builds on the technique developed in \cite[Proof of Theorem 7]{33}.
Recall, from \eqref{residual_form}, that
$\breve{r}_a$ has the form
$$
 \breve{r}_a(v): = \sum_{T \in \tria[a]} \int_T v \psi_a
 \phi_T+\sum_{e\in \intedges[a]} \int_e v \phi_e
 \quad (v \in H^1(\omega_a)),
 $$
for some $\phi_T \in P_{p_T-1}(T)$ and $\phi_e \in P_{p_a}(e)$.
  
Part (A) deals with the first term of $\check{r}_a$ whereas part (B)
handles the second one.
In fact, in (A) we use $\sup_{p \in \N} C^{(1)}_{p,q(p)}<\infty$
to construct $r_T \in H_*^1(\omega_a)'$ for each $T \in \tria[a]$ such that
\begin{equation} \label{x2}
\|r_T\|_{H_*^1(\omega_a)'} \lesssim \|\breve{r}_a\|_{(H_\ast^1(\omega_a)\cap\prod_{T \in \tria[a]} \P_{p_a+q(p_a)}(T))'},
\end{equation}
and decompose $\breve{r}_a$ as follows
$$
r_a^{(0)}:=\breve{r}_a+\sum_{T \in \tria[a]} r_T
$$
with $r_a^{(0)}$ satisfying
\begin{align} \label{x3}
r_a^{(0)}(\1)&=\breve{r}_a(\1),\\ \label{x4}
r_a^{(0)}(v)&=\sum_{e \in \intedges[a]} \int_e \phi^{(0)}_e v \quad\text{for some } \phi^{(0)}_e \in \P_{p_a}(e).
\end{align}

We next use $\sup_{p \in \N}\max\big(C^{(2)}_{p,q(p)},C^{(3)}_{p,q(p)}\big)<\infty$
in (B) to construct an $r_a^{(i)} \in H^1_*(\omega_a)'$ for each
$i=1,\ldots,n_a-1$ with $n_a:=\# \tria[a]$, such that
$r_a^{(n_a-1)}=0$ and
\begin{equation} \label{x5}
\|r_a^{(i)}-r_a^{(i-1)}\|_{H_*^1(\omega_a)'} \lesssim \|r^{(i-1)}_a\|_{(H_\ast^1(\omega_a)\cap\prod_{T \in \tria[a]} \P_{p_a+q(p_a)}(T))'}.
\end{equation}

Clearly, inequalities \eqref{x2}, \eqref{x5} imply that 
\begin{equation}\label{ra}
\begin{aligned}
\|r^{(0)}_a\|_{(H_\ast^1(\omega_a)\cap\prod_{T \in \tria[a]} \P_{p_a+q(p_a)}(T))'} &\lesssim \|\breve{r}_a\|_{(H_\ast^1(\omega_a)\cap\prod_{T \in \tria[a]} \P_{p_a+q(p_a)}(T))'}\\
\|r^{(i)}_a\|_{(H_\ast^1(\omega_a)\cap\prod_{T \in \tria[a]} \P_{p_a+q(p_a)}(T))'} &\lesssim \|r^{(i-1)}_a\|_{(H_\ast^1(\omega_a)\cap\prod_{T \in \tria[a]} \P_{p_a+q(p_a)}(T))'},
\end{aligned}
\end{equation}
respectively. Therefore, writing
\[
\breve{r}_a=-\sum_{T \in \tria[a]} r_T+\sum_{i=1}^{n_a-1}
\blue{\big(}r^{(i-1)}_a-r^{(i)}_a \blue{\big)}
\]
and combining \eqref{x2}, \eqref{x5},
and \eqref{ra} gives the asserted estimate \eqref{red-sat}.
\medskip

{\it Part (A): Bulk residual.} Lemma~\ref{internal} shows that there
exists a $\vec{\sigma}^{(1)}_T \in \RT_{p_T}(T)$, for each $T \in\tria[a]$,
such that $\divv \vec{\sigma}^{(1)}_T =\psi_a \phi_T$ and
\begin{equation}\label{sigma-1}
\|\vec{\sigma}_T^{(1)}\|_{T} \lesssim \|\psi_a \phi_T\|_{H^1_0(T)'}.
\end{equation}
For each $T \in \tria[a]$, we will construct a $\vec{\sigma}_T^{(2)} \in \RT_{p_T}(T)$ with $\divv \vec{\sigma}^{(2)}_T=0$,
$\vec{\sigma}_T^{(2)}\cdot \vec{n}_{T}=- \vec{\sigma}_T^{(1)} \cdot
\vec{n}_{T}$ on $e_T=e_{T,\omega_a}:=\partial T \cap \partial\omega_a$
(see Figure \ref{fig1}), and
\begin{equation}\label{sigma-2}
\|\vec{\sigma}^{(2)}_T\|_T \lesssim
\|\psi_a \phi_T\|_{H_{0,\partial T \setminus \{e_T\}}^1(T)'}.
\end{equation}
Then putting
$$
r_T(v):=\langle \vec{\sigma}_T^{(1)}+\vec{\sigma}_T^{(2)},\nabla v\rangle_T
\quad
(v \in H^1_*(\omega_a)),
$$
obviously $r_T(\1)=0$, whence \eqref{x3} is valid, whereas
integration by parts
$$
r_T(v)=-\int_T v \psi_a \phi_T+\int_{\partial T \setminus e_T} v (\vec{\sigma}_T^{(1)}+\vec{\sigma}_T^{(2)})\cdot \vec{n}_T
$$
shows \eqref{x4} for suitable $\phi_e^{(0)} \in \P_{p_a}(e)$.
Finally, \eqref{sigma-1} and \eqref{sigma-2} yield
\begin{align*}
\|r_T\|_{H_*^1(\omega_a)'} &\leq
\|\vec{\sigma}_T^{(1)}+\vec{\sigma}_T^{(2)}\|_T
\\
& \lesssim \|\psi_a \phi_T\|_{H_{0,\partial T \setminus \{e_T\}}^1(T)'}
\lesssim \|\psi_a \phi_T\|_{(H_{0,\partial T \setminus \{e_T\}}^1(T)\cap \P_{p_T+q(p_T)})'},
\end{align*}
where for the last inequality we have applied
$\sup_{p  \in \N} C^{(1)}_{p,q(p)}<\infty$.
To derive \eqref{x2} it remains to prove
\[
\|\psi_a \phi_T\|_{(H_{0,\partial T \setminus \{e_T\}}^1(T)\cap \P_{p_T+q(p_T)})'} \leq 
\|\breve{r}_a\|_{(H_\ast^1(\omega_a)\cap\prod_{T \in \tria[a]} \P_{p_a+q(p_a)}(T))'}.
\]
We proceed as follows: for any
$v \in H_{0,\partial T \setminus \{e_T\}}^1(T)\cap \P_{p_T+q(p_T)}$,
denote its zero extension to $\omega_a$ again by $v$ and set
$\bar{v}:=v-\vol(\omega_a)^{-1}\int_{\omega_a} v
\in H_\ast^1(\omega_a)\cap\prod_{T \in \tria[a]} \P_{p_a+q(p_a)}(T)$.
We then have $\int_T v \psi_a \phi_T = \breve{r}_a(v)=\breve{r}_a(\bar{v})$
thanks to $\breve{r}_a(\1)=0$, while obviously
$\|\nabla v\|_T=\|\nabla \bar{v}\|_{\omega_a}$. This proves the
desired estimate.

It remains to construct $\vec{\sigma}^{(2)} \in \RT_{p_T}(T)$ as required.
In view of Lemma~\ref{boundary}, it is sufficient to construct some $\vec{\tau}_T \in H(\divv;T)$ with 
$\divv \vec{\tau}_T =0$,
$\vec{\tau}_T\cdot \vec{n}_{T}=- \vec{\sigma}_T^{(1)} \cdot \vec{n}_{T}$ on $e_T$, and $\|\vec{\tau}_T\|_T \lesssim \|\psi_a \phi_T\|_{H_{0,\partial T \setminus \{e_T\}}^1(T)'}$.
This $\vec{\tau}_T$ can be chosen as $\nabla w_T$ with
\[
\Delta w_T=0
\quad \text{ in }T,
\qquad
w_T=0 \quad \text{ on }\partial T \setminus \{e_T\},
\qquad
\frac{\partial w_T}{\partial \vec{n}_{T}}=-\vec{\sigma}_T^{(1)} \cdot\vec{n}_{T}
\quad\text{ on } e_T.
\]
In fact, since
\[
\|\nabla w_T\|_T=
\Big\|v \mapsto \int_{e_T} v \vec{\sigma}^{(1)}\cdot
\vec{n}_{T} \Big\|_{H_{0,\partial T \setminus \{e_T\}}^1(T)'}
\]
and
\[
\int_{e_T} v \vec{\sigma}^{(1)}\cdot\vec{n}_{T} =
\int_{T} v \psi_a \phi_T+\vec{\sigma}^{(1)}_T \cdot \nabla v,
\]
by integration by parts, we deduce
\[
\|\nabla w_T\|_T 
\leq \|\psi_a \phi_T\|_{H_{0,\partial T \setminus \{e_T\}}^1(T)'}+\|\vec{\sigma}^{(1)}_T\|_T
\lesssim \|\psi_a \phi_T\|_{H_{0,\partial T \setminus \{e_T\}}^1(T)'}.
\]

\medskip

{\it Part (B): Edge residual}. Consider the notations as indicated in Figure~\ref{fig1}.
\begin{figure}[h]
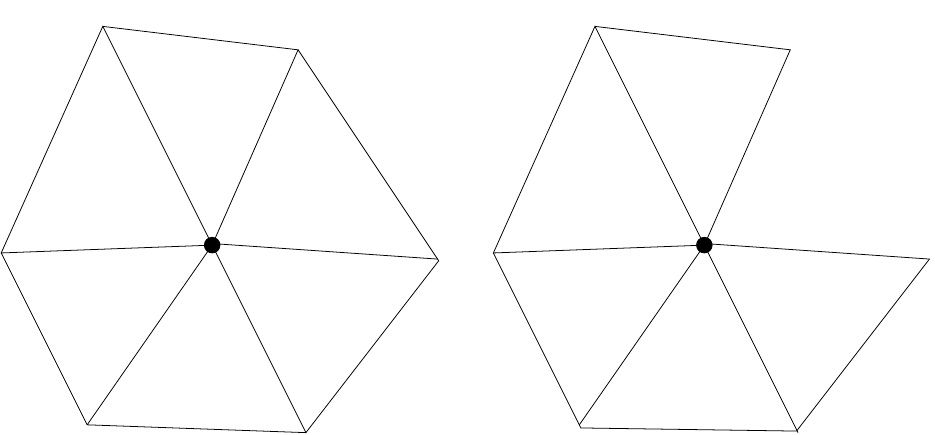
\caption{Enumeration of triangles and edges in $\tria[a]$ for the case $a \in \intnodes$ (left) or  $a \in \bdrnodes$ (right).}
\label{fig1}
\end{figure}
For $i=1,\ldots,n_a-1$, we will construct $r_a^{(i)}\in H^1_*(\omega_a)$, $r_a^{(n_a-1)}=0$, so that \eqref{x5} holds
\begin{align}\label{x6}
  r_a^{(i)}(\1) = r_a^{(i-1)}(\1),
\end{align}
and for some $\phi^{(i)}_{e_j} \in \P_{p_a}(e_j)$
\begin{equation}\label{x7}
  r_a^{(i)}(v) = \sum_{j=i+1}^{n_a} \int_{e_j} \phi^{(i)}_{e_j} v \quad
  (v \in H^1_*(\omega_a)).
\end{equation}

For $a \in \bdrnodes$, $v \mapsto \int_{e_1} \phi^{(0)}_{e_1}v$ is zero as element of  $H_*^1(\omega_a)'$ as $e_1 \subset \partial \Omega$ 
and $v=0$ on $\partial \Omega$. We can thus take $r_a^{(1)}=r_a^{(0)}$.

For $a \in \intnodes$, we will construct a
$\vec{\sigma}_{T_1} \in \RT_{p_a}(T_1)$ with $\divv
\vec{\sigma}_{T_1}=0$, $\vec{\sigma}_{T_1} \cdot \vec{n}_{T_1}=0$ on $e_{T_1}$,
$\vec{\sigma}_{T_1} \cdot \vec{n}_{T_1}=-\phi^{(0)}_{e_1}$ on $e_1$,
with $\phi_{e_1}^{(0)}$ introduced in \eqref{x4}, and
\begin{equation}\label{sigma-T1}
  \|\vec{\sigma}_{T_1}\|_{T_1} \lesssim \Big\|v \mapsto \int_{e_1}
  \phi^{(0)}_{e_1} v \Big\|_{H^1_{0,e_2}(T_1)'}.
\end{equation}
We define
$$
r_a^{(1)}(v):=r_a^{(0)}(v)+ \langle \vec{\sigma}_{T_1},\nabla v\rangle_{T_1}
\quad (v \in H^1(\omega_a)).
$$
Then $r_a^{(1)}(\1)=r_a^{(0)}(\1)$, $r^{(1)}_a(v)= \int_{e_2}
(\phi^{(0)}_{e_2} + \vec{\sigma}_{T_1} \cdot \vec{n}_{T_1})
v+\sum_{i=3}^{n_a} \int_{e_i} \phi^{(0)}_{e_i} v$ for $v\in H^1(\omega_a)$,
and thanks to $\sup_{p \in \N} C^{(2)}_{p,q(p)}<\infty$,
\begin{align*}
\|r_a^{(1)}-r_a^{(0)}\|_{H_*^1(\omega_a)'} &\leq \|\vec{\sigma}_{T_1}\|_{T_1} \lesssim
\Big\|v \mapsto \int_{e_1} \phi^{(0)}_{e_1} v \Big\|_{H^1_{0,e_2}(T_1)'}
\\ &\lesssim
\Big\|v \mapsto \int_{e_1} \phi^{(0)}_{e_1} v \Big\|_{(H^1_{0,e_2}(T_1) \cap \P_{p_a+q(p_a)})'}.
\end{align*}
We observe that in order to show \eqref{x5} for $i=1$ we still need to prove
\[
\Big\|v \mapsto \int_{e_1} \phi^{(0)}_{e_1} v \Big\|_{(H^1_{0,e_2}(T_1) \cap \P_{p_a+q(p_a)})'}
\lesssim 
\|r_a^{(0)}\|_{(H_\ast^1(\omega_a)\cap\prod_{T \in \tria[a]}
  \P_{p_a+q(p_a)}(T))'}.
\]
To do this, we first consider a simple affine transformation that makes
$T_{n_a}$ a reflection of $T_1$ across $e_1$, and extend boundedly by means of reflection
$v \in H^1_{0,e_2}(T_1) \cap \P_{p_a+q(p_a)}$ to a function
$\bar{v} \in H_{0,e_2\cup e_{n_a}}^1(T_1 \cup T_{n_a}) \cap \prod_{i \in \{1,n_a\}}
\P_{p_a+q(p_a)}(T_i)$.
We next identify $\bar{v}$ with its zero extension to the rest of
$\omega_a$ and set
$\bar{\bar{v}}:=\bar{v}-\vol(\omega_a)^{-1}\int_{\omega_a}\bar{v} \in
H_\ast^1(\omega_a)\cap\prod_{T \in \tria[a]} \P_{p_a+q(p_a)}(T)$.
Thanks to $r_a^{(0)}(\1)=0$ we have $\int_{e_1} \phi^{(0)}_{e_1}
v=r_a^{(0)}(\bar{\bar{v}})$ and $\|\nabla \bar{\bar{v}}\|_{\omega_a}
\lesssim \|\nabla v\|_T$.
This proves the desired inequality.

It remains to construct the aforementioned vector field
$\vec{\sigma}_{T_1}$. In view of Lemma~\ref{boundary}, it is
sufficient to construct another vector field $\vec{\tau}_{T_1} \in H(\divv;T_1)$ with $\divv \vec{\tau}_{T_1}=0$, $\vec{\tau}_{T_1} \cdot \vec{n}_{T_1}=0$ on $e_{T_1}$, $\vec{\tau}_{T_1} \cdot \vec{n}_{T_1}=-\phi^{(0)}_{e_1}$ on $e_1$, and $\|\vec{\tau}_{T_1}\|_{T_1} \lesssim \|v \mapsto \int_{e_1} \phi^{(0)}_{e_1} v\|_{H^1_{0,e_2}(T_1)'}$.
This $\vec{\tau}_{T_1}$ can be chosen to be $\nabla w_{T_1}$,
with $w_{T_1}$ being harmonic in $T_1$ and
\[
w_{T_1} = 0 \quad\text{ on } e_2,
\qquad
\frac{\partial w_{T_1}}{\partial \vec{n}_{T_1}}=0 \quad\text { on } e_{T_1},
\qquad \frac{\partial w_{T_1}}{\partial
  \vec{n}_{T_1}}=-\phi^{(0)}_{e_1}
\quad\text{ on } e_1.
\]
This function satisfies
$\|\nabla w_{T_1}\|_{T_1}=\|v \mapsto \int_{e_1} \phi^{(0)}_{e_1} v\|_{H^1_{0,e_2}(T_1)'}$,
which is \eqref{sigma-T1}.

The preceding procedure to construct $r_a^{(1)}$ (for $a \in \intnodes$) can be employed for $a \in \nodes=\intnodes \cup \bdrnodes$  to construct $r_a^{(2)},\ldots,r_a^{(n_a-2)}$, where \eqref{x5}, \eqref{x6}, and \eqref{x7} will be realized.

There remains the final case $i=n_a-1$. For $a \in \bdrnodes$, the same procedure can be applied another time, which yields \eqref{x5} for $i=n_a-1$, whereas \eqref{x7} shows that the resulting $r_a^{(n_a-1)}$ is zero as a functional on $H^1_*(\omega_a)$.

For $a \in \intnodes$, we set $r_a^{(n_a-1)}:=0$ as anticipated. 
In view of \eqref{x7}, as well as \eqref{x6} and \eqref{x3}, we deduce
\[
\sum_{j=n_a-1}^{n_a} \int_{e_j}
\phi_{e_j}^{(n_a-2)}=r_a^{(n_a-2)}(\1)=\breve{r}_a(\1)=0.
\]
We can thus apply $\sup_{p \in \N} C^{(3)}_{p,q(p)}<\infty$ to
$r_a^{(n_a-2)}=r_a^{(n_a-2)}-r_a^{(n_a-1)}$ to obtain
\begin{align*} 
\|r_a^{(n_a-2)}\|_{H^1_*(\omega_a)'} \leq \|r_a^{(n_a-2)}\|_{H^1_*(T_{n_a-1})'} \lesssim \|r_a^{(n_a-2)}\|_{(H^1_*(T_{n_a-1})\cap \P_{p_a+q(a)})'},
\end{align*}
where the first inequality follows from $r_a^{(n_a-2)}(\1)=0$, and the
fact that for each $v \in H^1_*(\omega_a)$, $\|\nabla
(v|_T-\vol(T)^{-1}\int_T v)\|_T \leq \|\nabla v\|_{\omega_a}$.
It remains to prove the estimate
\begin{equation}\label{lastinequality}
\|r_a^{(n_a-2)}\|_{(H^1_*(T_{n_a-1})\cap \P_{p_a+q(a)})'}
 \lesssim  
\|r_a^{(n_a-2)}\|_{(H_\ast^1(\omega_a)\cap\prod_{T \in \tria[a]} \P_{p_a+q(p_a)}(T))'}.
\end{equation}
We distinguish between $n_a$ even and odd. In the former case,
we use
that there exists a bounded extension of $H^1_*(T_{n_a-1})\cap \P_{p_a+q(p_a)}$ to $H^1_*(\omega_a)\cap \prod_{T \in \tria[a]} \P_{p_a+q(p_a)}(T)$.
Indeed, such an extension can be constructed by repeated reflections (modulo simple affine transformations) over $e_{n_a-1},\ldots, e_1$. 

For $n_a$ odd, instead, we follow the same procedure of reflecting a given $v \in H^1_*(T_{n_a-1})\cap \P_{p_a+q(p_a)}$
over $e_{n_a-1},\ldots, e_2$.
Then on the last triangle $T_{n_a}$, we are left with the problem
of finding an $H^1(T_{n_a})$-bounded extension of given 
$\sigma_i \in H^{\frac{1}{2}}(e_i) \cap \P_{p_a+q(p_a)}$ for $i \in
\{1,n_a\}$, to a polynomial of degree $p_a+q(p_a)$ on $T_{n_a}$.
Since the $\sigma_i$'s connect continuously at
$\{a\}=e_{1} \cap e_{n_a}$,
because of the structure of repeated reflections,
suitable extensions are known from the literature; see e.g. \cite[Lemma 7.2]{18.645}.
The resulting extension $\bar{v} \in H^1(\omega_a)\cap \prod_{T \in \tria[a]} \P_{p_a+q(p_a)}(T)$ of $v$ constructed in this way satisfies $\|\nabla \bar{v}\|_{\omega_a} \lesssim \|\nabla v\|_{T_{n_a-1}}$.
Finally, setting $\bar{\bar{v}}:=\bar{v}-\vol(\omega_a)^{-1}\int_{\omega_a} \bar{v} \in H_*^1(\omega_a)\cap \prod_{T \in \tria[a]} \P_{p_a+q(p_a)}(T)$,
we have $\|\nabla \bar{\bar{v}}\|_{\omega_a} \lesssim \|\nabla
v\|_{T_{n_a-1}}$, as well as
$r_a^{(n_a-2)}(\bar{\bar{v}})=r_a^{(n_a-2)}(v)$, thanks to $r_a^{(n_a-2)}(\1)=0$. 
 This completes the proof of \eqref{lastinequality}, and with that the proof of the theorem.
\end{proof}

\begin{remark}[one dimensional case: star indicator]
A simplified version of the proof of Theorem~\ref{thm2} shows that the corresponding result in one dimension is valid under the conditions that for some $q=q(p)$,
\begin{align} \label{first}
\sup_{p \in \N} & \sup_{0 \neq \phi \in \P_{p-1}(-1,1)}  \frac{\|(1-x)\phi\|_{H^1_{0,\{-1\}}(-1,1)'}}{\|(1-x)\phi\|_{(H^1_{0,\{-1\}}(-1,1) \cap \P_{p+q})'}}<\infty, \\
\label{second}
\sup_{p \in \N} &\,\,\frac{\|v \mapsto v(1)\|_{H^1_{0,\{-1\}}(-1,1)'}}{\|v \mapsto v(1)\|_{(H^1_{0,\{-1\}}(-1,1) \cap \P_{p+q})'}}  < \infty.
\end{align}

The numerator in \eqref{second} is equal to $\|u'\|_{(-1,1)}$, where
$u$ solves
\[
-u''=0 \quad\text{ on } (-1,1), \qquad u(-1)=0, \quad u'(1)=1,
\]
whereas the denominator is the $L_2(-1,1)$-norm of the
derivative of its Galerkin approximation from $H^1_{0,\{-1\}}(-1,1)
\cap \P_{p+q}$. Since $u(x)=x+1$ implies $u \in \P_1(-1,1)$,
\eqref{second} is obviously true even for $q=0$. \medskip

Similarly, the numerator of \eqref{first} is equal to $\|u'\|_{(-1,1)}$, where $u=u(\phi)$ solves
\[
-u'' = (1-x)\phi \quad  \text{ on } (-1,1),
\qquad
 u(-1) =  u'(1) = 0,
\]
whereas the denominator is the $L_2(-1,1)$-norm of the derivative of its Galerkin approximation $u_{p+q}=u_{p+q}(\phi)$ from $H^1_{0,\{-1\}}(1,-1) \cap \P_{p+q}$.
Since $u \in \P_{p+2}(-1,1)$, indeed $u(x)= -\int_{-1}^x \int_y^1
(1-z)\phi(z) \,dz\,dy$, obviously \eqref{first} is valid for
$q=2$. This is consistent with the derivation in \cite[Section 4.3]{35.99}.

Let us now take $q=1$, and investigate whether  \eqref{first} is still valid.
Since $H^1_{0,\{-1\}}(-1,1) \cap \P_{p+1} \rightarrow \P_p:v \mapsto
v'$ is surjective, $u_{p+1}'$ is the $L_2(-1,1)$-best approximation to
$u'$ from $\P_p$. Moreover, since $u'\in \P_{p+1}$, we have that
$(u-u_{p+1})'$ is a multiple of the 
Legendre polynomial of degree $p+1$. Since this polynomial does not
vanish at $1$, but $u'$ does, apparently $(u-u_{p+1})'\neq u'$,
whence $u_{p+1} \neq 0$.
From $\P_{p-1}(-1,1)$ being finite dimensional, we conclude that for any 
{\em fixed} $p$, 
$$
\sup_{0 \neq \phi \in \P_{p-1}(-1,1)} \frac{\|(1-x)\phi\|_{H^1_{0,\{-1\}}(-1,1)'}}{\|(1-x)\phi\|_{(H^1_{0,\{-1\}}(-1,1) \cap \P_{p+1})'}}<\infty.
$$
Below we study the question whether this holds {\em uniformly} in $p$, i.e., whether \eqref{first} is valid for $q=1$.

Let $\ell_n$ denote the $L_2(-1,1)$-normalized Legendre polynomial
of degree $n$. Exploiting that $(u-u_{p+q})'$ and $u_{p+q}'$ are
$L_2(-1,1)$-orthogonal, we deduce
$\|u_{p+1}'\|_{(-1,1)}^2$ $=\|u'\|_{(-1,1)}^2-|\langle u',\ell_{p+1}
\rangle_{(-1,1)}|^2$, or equivalently
\[
\frac{\|u'\|_{(-1,1)}}{\|u_{p+1}'\|_{(-1,1)}}=\Big(1-\frac{|\langle u',\ell_{p+1}
\rangle_{(-1,1)}|^2}{\|u'\|_{(-1,1)}^2}\Big)^{-\frac{1}{2}}.
\]
So saturation uniformly in $p$ holds for $q=1$ if and only if 
$$
\sup_{p} \rho_p<1, \text{ with }\rho_p:=\sup_{\phi \in \P_{p-1}(-1,1)} \frac{|\langle u',\ell_{p+1} \rangle_{(-1,1)}|}{\|u'\|_{(-1,1)}}.
$$

Since $\{\ell_1',\cdots,\ell_p'\}$ is a basis for $\P_{p-1}(-1,1)$, we may write $\phi=\sum_{i=1}^p c_i \ell'_i$, where $\vec{c}:=(c_i)_i$ runs over $\R^p$.
With
$$
\alpha_i:={\textstyle \frac{i\sqrt{2i+1}}{(2i+1)\sqrt{2i+3}}}, \quad \beta_i:={\textstyle \frac{(i+1)\sqrt{2i+1}}{(2i+1)\sqrt{2i-1}}},
$$
it holds that
$$
\alpha_i \ell_{i+1}'(x)=x \ell_i'(x)-\beta_i \ell_{i-1}'(x)\quad (i \geq 1),
$$
from which we infer that $(1-x)\phi(x)=\sum_{i=1}^{p+1}d_i \ell_i'(x)$, where $d_{p+1}=-\alpha_p c_p$, $d_{p}=c_p-\alpha_{p-1} c_{p-1}$, and for $i=p-1,\ldots,1$, $d_i=-\beta_{i+1} c_{i+1}+c_i-\alpha_{i-1} c_{i-1}$ (with $c_0:=0$).
Writing these relations as $\vec{d}=T \vec{c}$, where $T \in \R^{(p+1)\times p}$, we conclude that $\vec{d}$ runs over $\ran T=\{\vec{d} \in \R^{p+1}\colon \vec{d}^\top \vec{v}=0\}$ when $\ker T^\top=\Span\{\vec{v}\}$. This $\vec{v}=(v_1,\ldots,v_{p+1})$ can be found as the first $p+1$ elements of
$$
v_1=1,\quad v_2=\alpha_1^{-1},\quad v_i=\alpha_{i-1}^{-1}(v_{i-1}-\beta_{i-2} v_{i-2}) \quad (i=3,4,\ldots).
$$ 

In view of the expression for $(1-x)\phi(x)$,
$u'(x) = \int_1^x (1-z)\phi(z)\,dz$, and $\ell_i(1)=\sqrt{\frac{2i+1}{2}}$, 
we have $u'=\sum_{i=1}^{p+1} d_i \ell_i-\sum_{i=1}^{p+1} \sqrt{2i+1}\, d_i \ell_0$, and so
$$
 \rho^2_p=\sup_{\{\vec{d} \in \R^{p+1}\colon \vec{d}^\top \vec{v}=0\}} \frac{d_{p+1}^2}{\sum_{i=1}^{p+1} d_i^2+\big(\sum_{i=1}^{p+1} \sqrt{2i+1}\, d_i \big)^2}.
$$
For computing this supremum, it is sufficient to consider $d_{p+1}=1$. Setting $\vec{e}:=(d_i)_{1 \leq i \leq p}$, $\vec{w}:=(v_i)_{1 \leq i \leq p}$ and  $\vec{g}:=(\sqrt{2i+1})_{1 \leq i \leq p}$, we find that
$$
\rho^2_p=\Big(2p+4+\min_{\{\vec{e}\in \R^p\colon\vec{e}^\top \vec{w}=-v_{p+1}\}} ((I+\vec{g} \vec{g}^\top) \vec{e})^\top \vec{e}+2 \sqrt{2p+3}\,  \vec{g}^\top \vec{e}\Big)^{-1}.
$$
 The minimizer $\vec{e}$ can be computed as the solution of the saddle-point problem
 $$
 \left[\begin{array}{cc} I+\vec{g} \vec{g}^\top & \vec{w}\\\vec{w}^\top & 0\end{array}\right]
 \left[\begin{array}{c} \vec{e} \\ \lambda \end{array}\right]=
  \left[\begin{array}{c} -\sqrt{2p+3}\,\vec{g} \\ -v_{p+1} \end{array}\right],
 $$
 which gives
 \begin{align*}
 \lambda&=\big(\vec{w}^\top (I+\vec{g} \vec{g}^\top)^{-1} \vec{w}\big)^{-1} \big(v_{p+1}-\vec{w}^\top(I+\vec{g} \vec{g}^\top)^{-1} \sqrt{2p+3}\,\vec{g}\big),\\
 \vec{e}&=-(I+\vec{g} \vec{g}^\top)^{-1}(\sqrt{2p+3}\,\vec{g}+\lambda \vec{w}).
 \end{align*}%
   The computed values of $\rho^2_p$ given in Table~\ref{table1}
   \begin{table}[h]
  \begin{tabular}{|c|c|c|c|}
  \hline
  $p$ & $10$ & $100$ & $10000$\\
  \hline
  $\rho^2_p$ & 0.5719 & 0.9402 & 0.9994\\
 \hline
  \end{tabular}
 \medskip
  \caption{Computed values of ${\rho}^2_p$.}
  \label{table1}
  \end{table}
indicate that $\lim_{p \rightarrow \infty}  {\rho}_p=1$
which means that \eqref{first} is {\em not} valid for $q=1$. Apparently this is a price to be paid for the localization of the residuals using the partition of unity defined by the linear hats.
  Indeed, in this one-dimensional setting a localization is possible that allows for $q=1$, which we discuss in the next remark.
\end{remark}

\begin{remark}[one dimensional case: element indicator]
Let $\tria~$ be a subdivision of an interval $\Omega$ into subintervals.
With $I_{\tria}$ being the continuous linear interpolant w.r.t. $\tria$, the splitting $v=I_{\tria} v+\sum_{T \in \tria} (v-I_{\tria}v)|_T$ gives rise to the following orthogonal decomposition of $H^1_0(\Omega)$:
$$
H^1_0(\Omega)=\bigoplus_{T \in \tria} H^1_0(T) \,\,\bigoplus\,\, (H^1_0(\Omega) \cap \prod_{T \in \tria} \P_1(T)).
$$

For any $u_{\tria} \in U_{\tria}$ that satisfies \eqref{linears}, it holds that $r_{\tria}|_{H^1_0(\Omega) \cap \prod_{T \in \tria} \P_1(T)}=0$, and so with $r_T:=r_{\tria}|_{H^1_0(T)}=v \mapsto \int_T (f+u_{\tria}'')v$, we have
$$
\|(u-u_{\tria})'\|_\Omega=\|r_{\tria}\|_{H^{-1}(\Omega)}=\sqrt{ \sum_{T \in \tria} \|r_T\|_{H^{-1}(T)}^2}.
$$

With $\breve{r}_T:=v \mapsto \int_T (Q_{p_T-1,T} f+u_{\tria}'')v$, it holds that
$\big| \sqrt{ \sum_{T \in \tria} \|r_T\|_{H^{-1}(T)}^2}-\sqrt{ \sum_{T \in \tria} \|\breve{r}_T\|_{H^{-1}(T)}^2}\big| \lesssim \osc(f,\tria)$, with this oscillation term defined as in Corollary~\ref{corol1}.
Finally, since the  solution of a Poisson problem on $T$ with right-hand side $Q_{p_T-1,T} f+u_{\tria}'' \in \P_{p_T-1}(T)$ and homogenous Dirichlet boundary conditions on $\partial T$ is in $ \P_{p_T+1}(T)$,
we conclude
that
$\|\breve{r}_T\|_{H^{-1}(T)}=\|\breve{r}_T\|_{(H^1_0(T)\cap \P_{p_T+1})'}$. So with this approach even `full' saturation is obtained by raising the local polynomial degree by only one.
\end{remark}

\section{Computing the constants defined in Theorem~\ref{thm2}} \label{Sconstants}
In this section we discuss a couple of techniques to compute approximations of the
constants $C_{p,q}^{(i)}$ defined in Theorem~\ref{thm2}, and we indicate some choices of functions $q=q(p)$ for which these constants appear to be bounded uniformly in $p$. 

We start by
observing that, applying Riesz' lifts, the constants $C_{p,q}^{(i)}$
satisfy the following relations:
\begin{align*}
C^{(1)}_{p,q}&=\sup_{0 \neq \phi \in \P_{p-1}(\check{T})} \frac{\|\nabla u^{(1)}\|_{\check{T}}}{\|\nabla u^{(1)}_{p+q}\|_{\check{T}}},\\
C^{(2)}_{p,q}&=\sup_{0 \neq \phi \in \P_p(\check{e}_1)} \frac{\|\nabla u^{(2)}\|_{\check{T}}}{\|\nabla u^{(2)}_{p+q}\|_{\check{T}}},\\
C^{(3)}_{p,q}&=\sup_{\{0 \neq \phi \in \prod_{i=1}^3 \P_p(\check{e}_i)\colon \int_{\partial \check{T}}\phi =0\}} \frac{\|\nabla u^{(3)}\|_{\check{T}}}{\|\nabla u^{(3)}_{p+q}\|_{\check{T}}}.
\end{align*}
Hereafter, the functions
\begin{align*}
u^{(1)}_{p+q}&=u^{(1)}_{p+q}(\phi) \in H^1_{0,\check{e}_1 \cup \check{e}_2}(\check{T}) \cap \P_{p+q}(\check{T}),\\
u^{(2)}_{p+q}&=u^{(2)}_{p+q}(\phi) \in H^1_{0,\check{e}_2}(\check{T}) \cap \P_{p+q}(\check{T}),\\
u^{(3)}_{p+q}&=u^{(3)}_{p+q}(\phi) \in H_*^1(\check{T}) \cap \P_{p+q}(\check{T})
\end{align*}
are respectively the solutions of
\begin{equation}\label{riesz}
\begin{aligned}
\int_{\check{T}} \nabla u^{(1)}_{p+q} \cdot \nabla v &= \int_{\check{T}} \psi \phi \, v \quad(v \in H^1_{0,\check{e}_1 \cup \check{e_2}}(\check{T}) \cap \P_{p+q}(\check{T})),\\
\int_{\check{T}} \nabla u^{(2)}_{p+q} \cdot \nabla v &= \int_{\check{e}_1} \phi \, v \quad(v \in H^1_{0,\check{e}_2}(\check{T}) \cap \P_{p+q}(\check{T})),\\
\int_{\check{T}} \nabla u^{(3)}_{p+q} \cdot \nabla v &= \int_{\partial \check{T}} \phi \, v \quad(v \in H_*^1(\check{T}) \cap \P_{p+q}(\check{T})),
\end{aligned}
\end{equation}
 and $u^{(i)}:=u^{(i)}_{\infty}(\phi)$
are the exact solutions (setting $\P_{\infty}(\check{T}):=L_2(\check{T})$).

Since the solutions $u^{(i)}$, and so the constants $C^{(i)}_{p,q}$,
cannot be computed exactly, 
we approximate them by solving the three Poisson problems above
with a polynomial degree $r \gg p+q$ and exploit the fact that
$\lim_{r\to\infty} u_r^{(i)} = u^{(i)}$ in the corresponding closed
subspace of $H^1(\check{T})$.
We thus compute the constants
\begin{align} \label{constant1}
C^{(1)}_{p,q,r}&=\sup_{0 \neq \phi \in \P_{p-1}(\check{T})} \frac{\|\nabla u_r^{(1)}\|_{\check{T}}}{\|\nabla u^{(1)}_{p+q}\|_{\check{T}}},\\ \label{constant2}
C^{(2)}_{p,q,r}&=\sup_{0 \neq \phi \in \P_p(\check{e}_1)} \frac{\|\nabla u_r^{(2)}\|_{\check{T}}}{\|\nabla u^{(2)}_{p+q}\|_{\check{T}}},\\ \label{constant3}
C^{(3)}_{p,q,r}&=\sup_{\{0 \neq \phi \in \prod_{i=1}^3 \P_p(\check{e}_i)\colon \int_{\partial \check{T}}\phi =0\}} \frac{\|\nabla u_r^{(3)}\|_{\check{T}}}{\|\nabla u^{(3)}_{p+q}\|_{\check{T}}},
\end{align}
and expect that for $r$ large enough the values of $C^{(i)}_{p,q}$
would stabilize thereby indicating convergence of $u_r^{(i)}$ to $u^{(i)}$.
This process is documented in the tables below and is not a hidden
saturation assumption because the degree $r$ is not a priori decided
but determined from computations. In Section \ref{S:comp-constants} we
describe an alternative procedure based on the polynomial
structure of the forcing terms that circumvents this limiting process.

The corresponding Galerkin problems
to find $u_r^{(i)}$ are implemented via suitable modal bases of
the Koorwinder-Dubiner type (see, e.g., \cite{CHQZ06},
  Sect. 2.9.1) on the reference triangle $\check{T}=\{x,y \geq -1,
\ x+y \leq 0\}$. Exploiting the warped-tensor-product structure of
these functions, all integrals in the stiffness matrices and right-hand sides
are computed through univariate
Gaussian quadratures that are exact within machine accuracy for their
(polynomial) integrands. The constants of interest are computed by solving
suitable generalized eigenvalue problems.

We prefer this approach over that in Section \ref{S:comp-constants}
because of its relative simplicity when the underlying polynomial
degree $p$ becomes large.

\subsection{Constant $C_{p, q, r}^{(1)}$} \label{S:constant-1}
The following computational results for $C_{p, q, r}^{(1)}$ in
Tables 2 and 3 show that $r=2(p+q)$ yields
stabilization (only for the largest value of $p+q$ in each Table we restrict to this single value of $r$). 
Moreover, the choice $q(p)=p$ gives full saturation
$C_{p,q,r}^{(1)}\approx 1$ whereas $q(p)=p/7$ is more practical and
still gives an acceptable level of saturation. Table 4
for $q(p)=4$ displays a moderate increase of the saturation constant 
$C_{p, q, r}^{(1)}$. 

{\small
\begin{table}[h]\label{T1:q=p}
\caption{\small Here $q = p$, i.e., $p+q=2p$. There is clear evidence of convergence for $r \rightarrow \infty$. Furthermore, the limit constants are uniformly bounded with respect to $p$, and even seem to converge to $1$ (`full' saturation) when $p \rightarrow \infty$.}
\begin{tabular}{ r | r | r | l }
$p$ & $p+q$ & $r$ & $C_{p, q, r}^{(1)}$ \\
\hline \hline
4 & 8 & 16 & 1.0072779439\\
 &  &  32& 1.0072781599 \\
 &  &  64 & 1.0072781600\\
  &  &  128 &    1.0072781600\\
\hline 
 8 & 16 & 32 &  1.0007015305\\ 
   &  & 64 & 1.0007015438 \\ 
      &  & 128 & 1.0007015438 \\ 
\hline      
16 & 32 & 64 &  1.0001682679 \\
& & 96 &  1.0001682633\\
& & 128 & 1.0001682680\\
\hline  
32 & 64 & 128 & 1.0000689675\\
\end{tabular}
\end{table}

\begin{table}[h]\label{T1:q=p/7}
\caption{\small Here $q=\frac17 p$, i.e., $p+q = \frac87 p$. 
Uniform boundedness with respect to $p$ is preserved, and even convergence to $1$ for $p \rightarrow \infty$ seems to be valid.}
\begin{tabular}{ r | r | r | l }
$p$ & $p+q$ & $r$ & $C_{p, q, r}^{(1)}$ \\
\hline \hline
14 & 16 &  32 & 10.109219047 \\ 
& &  64 &  10.109454622\\ 
& &  128 & 10.109454650\\
\hline
 28 & 32 & 64 &  1.6580711707 \\
  &  & 128 &   1.6580859228\\
\hline
56 & 64 &  128 &  1.3327470997\\
\end{tabular}
\end{table}

\begin{table}[h]\label{T1:q=4}
\caption{\small Here $q=4$, i.e., $p+q=p+4$. 
Only the largest value of $r$ used in computation is reported. In this case, the constants $C_{p, q,r}^{(1)}$ (slowly) increase for $p \rightarrow \infty$.}
\begin{tabular}{ r | r | r | l }
$p$ & $p+q$ & $r$ & $C_{p, q, r}^{(1)}$ \\
\hline \hline
4 & 8 & 32 &  1.0072781599\\
12 & 16 &  64&   1.1590636448 \\ 
28 & 32 &  96 &   1.6580856832 \\
60 & 64 &  128 &   2.7635533362\\
\end{tabular}
\end{table}
}

\subsection{Constant $C_{p, q, r}^{(2)}$} \mbox{}
The same comments of Section \ref{S:constant-1} are valid here, although stabilization occurs at larger values of $r$.

{\small
\begin{table}[h]
\caption{\small Here $q = p$, i.e., $p+q=2p$. There is clear evidence of convergence for $r \rightarrow \infty$. Furthermore, the limit constants are uniformly bounded with respect to $p$,
and even seem to converge to $1$ when $p \rightarrow \infty$.}
\begin{tabular}{ r | r | r | l }
$p$ & $p+q$ & $r$ & $C_{p, q, r}^{(2)}$ \\
\hline \hline
4 & 8 & 16 & 1.1500400619\\
 &  &  32 &  1.1608825787\\
 &  &  64 & 1.1616050286\\
  &  &  128 & 1.1616516366\\
\hline 
 8 & 16 & 32 & 1.0928924221 \\ 
   &  & 64 &  1.0992140060\\ 
      &  & 128 &  1.0996224599\\ 
\hline      
16 & 32 & 64 & 1.0708125134  \\
& & 96 &  1.0747936682\\
& & 128 & 1.0754714541\\
\hline  
32 & 64 & 128 & 1.0611369396\\
\end{tabular}
\end{table}

\begin{table}[h]
\caption{\small Here $q=\frac17 p$, i.e., $p+q = \frac87 p$. 
Uniform boundedness with respect to $p$, and even convergence to $1$ for $p \rightarrow \infty$ seem to be preserved.}
\begin{tabular}{ r | r | r | l }
$p$ & $p+q$ & $r$ & $C_{p, q, r}^{(2)}$ \\
\hline \hline   
14 & 16 &  32 & 2.6706917112 \\ 
& & 64 &   2.7805456663\\
& & 128 &   2.7877362832\\
\hline
 28 & 32 & 64 &   2.0554724235\\
 & & 128 &   2.1293823858\\
 \hline
56 & 64 &  128 &  1.9013521194\\
\end{tabular}
\end{table}

\begin{table}[h]
\caption{\small Here $q=4$, i.e., $p+q=p+4$. 
Only the largest value of $r$ used in computation is reported. The constants $C_{p, q,r}^{(2)}$ (slowly) increase for $p \rightarrow \infty$.
}
\begin{tabular}{ r | r | r | l }
$p$ & $p+q$ & $r$ & $C_{p, q, r}^{(2)}$ \\
\hline \hline
4 & 8 & 32 & 1.1608825787 \\
12 & 16 &  64& 1.5500093739   \\ 
28 & 32 &  96 &  2.1185446449 \\
60 & 64 &  128 & 2.7568423884 \\
\end{tabular}
\end{table}
}

\subsection{Constant $C_{p, q, r}^{(3)}$} \mbox{}
The same comments of Section \ref{S:constant-1} are valid here.

{\small
\begin{table}[h]
\caption{\small Here $q = p$, i.e., $p+q=2p$. There is clear evidence of convergence for $r \rightarrow \infty$. 
Furthermore, the limit constants are uniformly bounded with respect to $p$,
and even seem to converge to $1$ when $p \rightarrow \infty$.}
\begin{tabular}{ r | r | r | l }
$p$ & $p+q$ & $r$ & $C_{p, q, r}^{(3)}$ \\
\hline \hline
4 & 8 & 16 & 1.0316563321\\
&  &  32 & 1.0318040514\\
&  &  64 & 1.0318046947\\
&  &  128 & 1.0318046973\\
\hline 
 8 & 16 & 32 &  1.0135088572\\ 
   &  & 64 &  1.0135679473\\ 
      &  & 128 &  1.0135681920\\ 
\hline      
16 & 32 & 64 &  1.0081863729 \\
& & 96 &  1.0082192602\\
& & 128 & 1.0082204858\\
\hline  
32 & 64 & 128 & 1.0062046674\\
\end{tabular}
\end{table}

\begin{table}[h]
\caption{\small Here $q=\frac17 p$, i.e., $p+q = \frac87 p$. 
Uniform boundedness with respect to $p$, and even convergence to $1$ for $p \rightarrow \infty$ seem to be preserved.}
\begin{tabular}{ r | r | r | l }
$p$ & $p+q$ & $r$ & $C_{p, q, r}^{(3)}$ \\
\hline \hline   
14 & 16 &  32 &  2.2934830389\\ 
& & 64 &   2.3001147590\\
& & 128 &  2.3001422769 \\
\hline
 28 & 32 & 64 &   1.6814554754\\
 & & 128 &   1.6851850386\\
 \hline
56 & 64 &  128 & 1.5469935612 \\
\end{tabular}
\end{table}

\begin{table}[h]
\caption{\small Here $q=4$, i.e., $p+q=p+4$. 
Only the largest value of $r$ used in computation is reported. The constants $C_{p, q,r}^{(3)}$ (slowly) increase for $p \rightarrow \infty$.}
\begin{tabular}{ r | r | r | l }
$p$ & $p+q$ & $r$ & $C_{p, q, r}^{(3)}$ \\
\hline \hline
4 & 8 & 32 &  1.0318046947\\
12 & 16 &  64&   1.2576399758 \\ 
28 & 32 &  96 &   1.6850507900\\
60 & 64 &  128 &  2.2721068822\\
\end{tabular}
\end{table}
}

\subsection{Equivalent computable constants}\label{S:comp-constants}
%
We now exploit the polynomial structure of the forcing functions
$\phi$ in \eqref{riesz} to show that the quantities $\|\nabla u^{(i)}\|_{\check{T}}$
can be computed via suitable saddle point problems with Raviart-Thomas
elements.

The first two paragraphs of Part (A) of the proof of Theorem~\ref{thm2} show that, for a given $\phi \in P_{p-1}(\check{T})$, there exists a $\vec{\sigma} \in \RT_p(\check{T})$ with $\divv  \vec{\sigma} =\psi \phi$ on $\check{T}$, $\vec{\sigma} \cdot \vec{n}_{\check{T}}=0$ on $\check{e}_3$, and $\|\vec{\sigma}\|_{\check{T}} \lesssim 
\|\psi \phi\|_{H^1_{0,\check{e}_1 \cup \check{e}_2}(\check{T})'}$.
The first two properties show that for $v \in H^1_{0,\check{e}_1 \cup \check{e}_2}(\check{T})$, $\int_{\check{T}} \psi \phi v=-\int_{\check{T}} \vec{\sigma} \cdot \nabla v$, and so $\|\psi \phi\|_{H^1_{0,\check{e}_1 \cup \check{e}_2}(\check{T})'} \lesssim \|\vec{\sigma}\|_{\check{T}}$.
We conclude that for $\vec{\sigma} \in \RT_p(\check{T})$ with $\divv  \vec{\sigma} =\psi \phi$ on $\check{T}$, $\vec{\sigma} \cdot \vec{n}_{\check{T}}=0$ on $\check{e}_3$ and minimal $\|\vec{\sigma}\|_{\check{T}}$, it holds that 
$$
 \|\vec{\sigma}\|_{\check{T}} \eqsim \|\psi \phi\|_{H^1_{0,\check{e}_1 \cup \check{e}_2}(\check{T})'} (=\|\nabla u^{(1)}\|_{\check{T}}).
$$
This $\vec{\sigma}$ can be computed as the first component of $(\vec{\sigma},t) \in \RT_p(\check{T}) \cap H_{0,\check{e}_3}(\divv;\check{T}) \times \P_p(\check{T})$ that  solves the discrete saddle-point problem
\begin{equation*}\label{comp-const1}
\left\{
\begin{aligned}
\int_{\check{T}} \vec{\sigma} \cdot \vec{\tau}+ 
\int_{\check{T}} t \divv \vec{\tau}
&= 0 &  (\vec{\tau} \in \RT_{p}(\check{T}) \cap H_{0,\check{e}_3}(\divv;\check{T}) ),\\
\int_{\check{T}} r \divv \vec{\sigma} &=  \int_{\check{T}} r \psi \phi & (r \in \P_p(\check{T})).
\end{aligned}
\right.
\end{equation*}
(Here, as expected, $H_{0,\check{e}_3}(\divv;\check{T})$ are the $H(\divv;\check{T})$-functions with vanishing normal components on $\check{e}_3$).

Similarly, as follows by steps made in Part (B) of the proof of Theorem~\ref{thm2}, for $\phi \in \P_p(\check{e}_1)$ there exists a $\vec{\sigma} \in \RT_p(\check{T})$ with $\divv \vec{\sigma}=0$ on $\check{T}$, $\vec{\sigma}  \cdot \vec{n}_{\check{T}}=\phi$ on $\check{e}_1$, 
$\vec{\sigma}  \cdot \vec{n}_{\check{T}}=0$ on $\check{e}_3$,
and $\|\vec{\sigma}\|_{\check{T}} \lesssim 
\|v \mapsto \int_{\check{e}_1} \phi v\|_{H^1_{0,\check{e}_2}(\check{T})'}$.
The first three properties show that for $v \in H^1_{0,\check{e}_2}(\check{T})$, $\int_{\check{e}_1} \phi v=\int_{\check{T}}\vec{\sigma} \cdot \nabla v$, and so $\|v \mapsto \int_{\check{e}_1} \phi v\|_{H^1_{0,\check{e}_2}(\check{T})'} \lesssim \|\vec{\sigma}\|_{\check{T}}$.
We conclude that for  $\vec{\sigma} \in \RT_p(\check{T})$ with $\divv \vec{\sigma}=0$ on $\check{T}$, $\vec{\sigma}  \cdot \vec{n}_{\check{T}}=\phi$ on $\check{e}_1$, 
$\vec{\sigma}  \cdot \vec{n}_{\check{T}}=0$ on $\check{e}_3$ and $\|\vec{\sigma}\|_{\check{T}}$ minimal, it holds that 
$$
\|\vec{\sigma}\|_{\check{T}} \eqsim \|v \mapsto \int_{\check{e}_1} \phi v\|_{H^1_{0,\check{e}_2}(\check{T})'} (=\|\nabla u^{(2)}\|_{\check{T}}).
$$
This $\vec{\sigma}$ is the first component of $(\vec{\sigma},t,t_1) \in \RT_p(\check{T})\cap H_{0,\check{e}_3}(\divv;\check{T}) \times \P_p(\check{T})\times \P_p(\check{e}_1)$ that  solves
\begin{equation*}\label{comp-const2}
\left\{
\begin{aligned}
\int_{\check{T}} \vec{\sigma} \cdot \vec{\tau}+ 
\int_{\check{T}} t \divv \vec{\tau}+\int_{\check{e}_1} t_1 \vec{\tau}\cdot \vec{n}_{\check{T}}
&= 0 &  (\vec{\tau} \in \RT_{p}(\check{T})\cap H_{0,\check{e}_3}(\divv;\check{T})),\\
\int_{\check{T}} r \divv \vec{\sigma}+\int_{\check{e}_1} r_1 \vec{\sigma}\cdot \vec{n}_{\check{T}}  &=  \int_{\check{e}_1} r_1 \phi & (r \in \P_p(\check{T}),\,r_1\in \P_p(\check{e}_1)).
\end{aligned}
\right.
\end{equation*}

Finally, for $\phi \in \prod_{i=1}^3 \P_p(\check{e}_i)$ with $\int_{\partial \check{T}} \phi=0$, from Lemma~\ref{boundary} one infers that there exists a $\vec{\sigma} \in \RT_p(\check{T})$ with $\divv \vec{\sigma}=0$ on $\check{T}$, $\vec{\sigma}  \cdot \vec{n}_{\check{T}}=\phi$ on $\partial \check{T}$, and $\|\vec{\sigma}\|_{\check{T}} \lesssim \|v \mapsto \int_{\partial \check{T}} \phi v\|_{H_*^1(\check{T})'}$. For such a $\vec{\sigma}$ with minimal $\|\vec{\sigma}\|_{\check{T}}$, it holds that
$$
\|\vec{\sigma}\|_{\check{T}} \eqsim \|v \mapsto \int_{\partial \check{T}} \phi v\|_{H_*^1(\check{T})'} (=\|\nabla u^{(3)}\|_{\check{T}}).
$$

Using that $\RT_p(\check{T}) \cap H(\divv 0;\check{T})=\curl (P_{p+1}(\check{T})/\R)$, one infers that 
this $\vec{\sigma}$ is the first component of 
$(\vec{\sigma},t) \in \curl (P_{p+1}(\check{T})/\R) \times \prod_{i=1}^3\P_p(\check{e}_i)$ that  solves
\begin{equation*}\label{compu-const3}
\left\{
\begin{aligned}
\int_{\check{T}} \vec{\sigma} \cdot \vec{\tau}+ 
\int_{\partial \check{T}} t \vec{\tau}\cdot \vec{n}_{\check{T}}
&= 0 &  (\vec{\tau} \in \curl (P_{p+1}(\check{T})/\R)),\\
\int_{\partial \check{T}} r \vec{\sigma}\cdot \vec{n}_{\check{T}}  &=  \int_{\partial \check{T}} r \phi & (r \in \prod_{i=1}^3\P_p(\check{e}_i)).
\end{aligned}
\right.
\end{equation*}

\section{Conclusion} \label{Sconclusion}
We show a $p$-robust contraction property for $hp$-AFEM
whenever a $p$-robust {\it local} saturation property, expressed in
terms of negative norms of residuals on patches of
elements sharing a node (stars),
is valid for all marked stars using D\"orfler marking.
We reduce the question of $p$-robust local saturation for the
Poisson problem in two dimensions to three simpler Poisson problems
over the reference triangle $\check{T}$ with
three different interior or boundary forcing $\phi$ which are
polynomials of degrees $p-1$ or $p$ respectively. If $u^{(i)}(\phi)$ are the exact
solutions of these auxiliary problems in $\check{T}$ for $i=1,2,3$, and
$u_{p+q(p)}^{(i)}(\phi) \in \P_{p+q(p)}(\check{T})$ are the
corresponding discrete solutions with polynomial degree $p+q(p)$,
the question reduces to finding a function $q(p)$ such that the
following saturation constants are uniformly bounded in $p$:
\[
C^{(i)}_{p,q(p)} = \sup_\phi \frac{\|\nabla
  u^{(i)}(\phi)\|_{L_2}(\check{T})}{\|\nabla
  u^{(i)}_{p+q(p)}(\phi)\|_{L_2}(\check{T})}
\quad i=1,2,3.
\]
We provide computational evidence that a function of the form
$q(p)=\lceil \lambda p\rceil$ gives uniform saturation, namely $C^{(i)}_{p,p(q)}$
is bounded as a function of $p$, for any
constant $\lambda>0$. In contrast, we do not observe uniform
saturation for a function $q(p)$ constant. However, for $q(p)=4$ and all
$p \leq 60$, the constants $C^{(i)}_{p,p+4}$ do not exceed $2.77$,
which is still quite close to the ideal value 1.



\begin{thebibliography}{BCMP91}

\bibitem[BCMP91]{18.645}
I.~Babu{\v{s}}ka, A.~Craig, J.~Mandel, and J.~Pitk{\"a}ranta.
\newblock Efficient preconditioning for the p-version finite element method in
  two dimensions.
  \newblock {\em {SIAM J. Numer. Anal.}}, 28(3):624Ð--661, 1991.

\bibitem[BW85]{BW85}
R.E. Bank and A. Weiser.
\newblock  Some a posteriori error estimators for elliptic partial 
differential equations
\newblock {\em Math. Comp.}, 44:285--301, 1985.

\bibitem[BEK:96]{BEK:96}
F. A. Bornemann, B. Erdmann, and R. Kornhuber,
\newblock A posteriori error estimates for elliptic problems in
	two and three space dimensions,
\newblock  \emph{SIAM J. Numer. Anal.}, 33:1188--1204, 1996.
        
\bibitem[BD11]{69.1}
M.~B{\"u}rg and W.~D{\"o}rfler.
\newblock Convergence of an adaptive {$hp$} finite element strategy in higher
  space-dimensions.
\newblock {\em Appl. Numer. Math.}, 61(11):1132--1146, 2011.

\bibitem[BDD04]{21}
P.~Binev, W.~Dahmen, and R.~{DeV}ore.
\newblock Adaptive finite element methods with convergence rates.
\newblock {\em Numer. Math.}, 97(2):219 -- 268, 2004.

\bibitem[Bin15]{22.556}
P.~Binev.
\newblock Tree approximation for $hp$-adaptivity.
\newblock IMI Preprint Series 2015:07, 2015.

\bibitem[BPS09]{33}
D.~Braess, V.~Pillwein, and J.~Sch{\"o}berl.
\newblock Equilibrated residual error estimates are {$p$}-robust.
\newblock {\em Comput. Methods Appl. Mech. Engrg.}, 198(13-14):1189--1197,
  2009.

\bibitem[CM10]{45.496}
M.~Costabel and A.~McIntosh.
\newblock On {B}ogovski\u\i\ and regularized {P}oincar\'e integral operators
  for de {R}ham complexes on {L}ipschitz domains.
\newblock {\em Math. Z.}, 265(2):297--320, 2010.

\bibitem[CHQZ06]{CHQZ06}
C.~Canuto, M.~Y.Hussaini, A.~Quarteroni, and T.~A.~Zang.
\newblock {\sl Spectral Methods. Fundamentals in Single Domains.}
\newblock Springer 2016.

\bibitem[CNSV15] {35.99}
C.~Canuto, R.H.~Nochetto, R. Stevenson, and M.~Verani.
\newblock Convergence and optimality of $hp$-AFEM.
\newblock {\it Numer. Math.}, 10.1007/s00211-016-0826-x, 2016




\bibitem[DGS12]{64.145}
L.~Demkowicz, J.~Gopalakrishnan, and J.~Sch{\"o}berl.
\newblock Polynomial extension operators. {P}art {III}.
\newblock {\em Math. Comp.}, 81(279):1289--1326, 2012.

\bibitem[DN02]{DN02}
W. D\"orfler and R.H. Nochetto, 
\newblock Small data oscillation implies the saturation assumption,
with W. D\"orfler,
\newblock \emph{Numer. Math.} 91:1--12, 2002.

\bibitem[EV15]{70.8}
A.~Ern and M.~Vohral{\'{\i}}k.
\newblock Polynomial-degree-robust a posteriori estimates in a unified setting
  for conforming, nonconforming, discontinuous {G}alerkin, and mixed
  discretizations.
\newblock {\em SIAM J. Numer. Anal.}, 53(2):1058--1081, 2015.

\bibitem[GB86a]{77.55}
B.~Guo and I.~Babu{\v{s}}ka.
\newblock {The h-p version of finite element method, Part 1: The basic
  approximation results}.
\newblock {\em {Comp. Mech}}, 1:21--41, 1986.

\bibitem[GB86b]{77.56}
B.~Guo and I.~Babu{\v{s}}ka.
\newblock {The h-p version of finite element method, Part 2: General results
  and application}.
\newblock {\em {Comp. Mech}}, 1:203--220, 1986.

\bibitem[MW01]{MW01}
J.~M. Melenk and B.~I. Wohlmuth.
\newblock On residual-based a posteriori error estimation in {$hp$}-{FEM}.
\newblock {\it Adv. Comput. Math.}, 15(1-4):311--331, 2002.

\bibitem[Noch93]{Noch93}
R.H. Nochetto,
\newblock Removing the saturation assumption in a posteriori error
analysis,
\newblock {\emph Istit. Lombardo Accad. Sci. Lett. Rend. A},
127:67-82, 1993.

\bibitem[Voh16]{309}
M.~Vohral\'{i}k.
\newblock Polynomial-degree-robust estimates in three space dimensions.
\newblock {Presentation at MAFELAB 2016, Brunel University, Uxbridge, UK},
  2016.

\end{thebibliography}
\end{document}